\documentclass[a4paper,11pt,reqno]{amsart}
\usepackage{amssymb}
\usepackage{amsmath}
\usepackage{amscd}
\usepackage{stmaryrd}

\def\c{\gamma}

\def\ra{\rightarrow}
\def\lra{\longrightarrow}
\def\.{\cdot}

\def\OO{\mathrm{O}}

\def\nb{\nabla}

\def\beq{\begin{equation}}
\def\eeq{\end{equation}}
\def\bi{\begin{enumerate}}
\def\ei{\end{enumerate}}
\def\bea{\begin{eqnarray*}}
\def\eea{\end{eqnarray*}}
\def\beay{\begin{eqnarray}}
\def\eeay{\end{eqnarray}}
\def\ba{\begin{array}}
\def\ea{\end{array}}
\def\x{\times}

\def\e{\varepsilon}

\def\s{\sigma}

\def\r{\end{proof}}

\def\ot{\otimes}

\def\ci{\mathcal{C}^\infty}

\def\Lb{\Lambda}

\def\g{\gamma}
\def\dg{\dot\gamma}
\def\ddg{\ddot\gamma}
\def\dddg{\dddot\gamma}

\def\te{\theta}
\def\wtl{\widetilde}

\def \R{\mathbb{R}}

\def \T{\mathbb{T}}

\def\End{{\rm End}}
\def\tl{\tilde}

\def\9{[\! (}
\def\0{)\! ]}
\def\[{\pmb{[}}
\def\]{\pmb{]}}

\def\rank{\mathrm{rank}}

\def\dl{{\delta}}

\def\lra{\longrightarrow}

\def\we{\wedge}
\def\d{{\partial}}
\def\Ric{\mathrm{Ric}}

\def\be{\begin{equation}}

\def\ee{\end{equation}}

\def\scal{\rm{Scal}}
\def\tr{\mathrm{tr}}

\def\GL{\mathrm{GL}}

\def\CO{\mathrm{CO}}
\def\SO{\mathrm{SO}}

\def\so{\mathfrak{so}}
\def\co{\mathfrak{co}}

\def\ric{\mathrm{Ric}}
\def\scal{\mathrm{Scal}}

\def\sm{\smallsetminus}
\def\opr{\stackrel{\perp}{\oplus}}

\def\idt{\mathrm{Id}}

\newtheorem{prop}{Proposition}[section]
\newtheorem{cor}[prop]{Corollary}
\newtheorem{lem}[prop]{Lemma}
\newtheorem{thrm}[prop]{Theorem}

\theoremstyle{definition}
\newtheorem{defi}[prop]{Definition}
\theoremstyle{remark}
\newtheorem{rem}[prop]{Remark}
\newtheorem{example}[prop]{Example}

\def\hs{\mathrm{Hess}}
\def\M{\mathcal{M}}
\def\Lp{\mathcal{L}}
\def\mx{\mu}
\def\rs{\rho}

\def\weyl{\mathrm{Weyl}}
 
\title{Geodesics and submanifold structures 
  in conformal geometry}
\author{Florin Belgun}

\address{Florin Belgun\\ Fachbereich Mathematik\\ Universit{\"a}t
 Hamburg\\ Bundesstr. 55\\ D-20146 Hamburg, Germany}
 \email{florin.belgun@math.uni-hamburg.de}

\address{Institutul de Matematic\u a al Academiei Rom\^ane, P.O. Box
  1--764, RO 014700 Buchrest, Romania}

\begin{document}

\begin{abstract}
A conformal structure on a manifold $M^n$ induces natural second order
conformally invariant operators, called M\"obius and Laplace
structures, acting on specific  weight bundles of $M$, provided that
$n\ge 3$. By extending the notions of M\"obius and Laplace
structures to the case of surfaces and curves, we develop here the
theory of extrinsic conformal geometry for submanifolds, find
tensorial invariants of a conformal embedding, and use these
invariants to characterize various  forms of geodesic submanifolds.
 \bigskip


\end{abstract}

\maketitle

\section{Introduction}
The existence of a unique covariant derivative makes differential
calculus, including the concept of totally geodesic submanifolds (or,
more generally, the tensorial invariants of a Riemannian embedding) in
a Riemannian manifold straightforward.

On a conformal
manifold, where there is no such canonical covariant derivative
(there
is, however, a {\em Cartan connection} on an enlarged bundle, for
dimension at least 3, \cite{cap}), a
concept of {\em conformal geodesics} is also given \cite{ferr},
\cite{ga} starting from dimension 3 onwards.

These {\em conformal geodesics} are curves that are solutions of a 3rd order ODE
that depends on the conformal structure alone.

In this paper, we intend to characterize higher-dimensional
submanifolds that fulfil some {\em geodesic} properties in the
conformal setting and describe the geometric properties and invariants
of a conformal embedding.

To make the theory fully general, a first inconsistency of conformal
geometry has to be overcome: indeed, while in dimensions larger than 3
a conformal manifold admits an associated Cartan connection and is,
therefore, {\em rigid}, on curves, a conformal structure means just a
differential structure, and on surfaces, a conformal structure and an
orientation are equivalent to a complex structure -- both are examples
of {\em flexible structures}.

Here, we call a geometric structure on $M$ {\em
  rigid} if on every open set, the dimension of the space of infinitesimal
transformations (vector fields) preserving the given structure is
bounded by a number that depends only on $\dim M$ and the corresponding
structure. Otherwise it is {\em flexible}. Examples of rigid
structures are Riemannian metrics, conformal structures if $\dim M\ge
3$, $CR$ structures and, more generally, all structures that admit a canonical Cartan connection;
symplectic, complex, contact structures are examples of flexible
structures.

Using the concept of a {\em M\"obius
  structure}, defined by D. Calderbank as a linear second order
differential operator of a certain type \cite{ca}, and also using a {\em Laplace structure} (a variant of the {\em conformal
  Laplacian}) to rigidify a curve, we create a setting
conformal-M\"obius-Laplace for which the questions of submanifold
geometry can be studied without conditions on the dimension.

In particular, on a M\"obius
surface or a Laplace curve, the concept of a conformal (or rather
M\"obius, resp. Laplace) geodesic is well-defined, and denotes, as well, the
family of curves that are solutions to a 3rd order ODE. Technically, these equations are given, in terms of a conformal
covariant derivative ({\em Weyl structure}) and of its associated
{\em Schouten-Weyl tensor}, and this tensor is defined, in low
dimensions, precisely by the corresponding additional structure
(M\"obius, resp. Laplace). The invariants of, and induced structures
on  a conformal embedding
are also defined in terms of these Schouten-Weyl tensors, the
distinction between them being the following:
\begin{itemize}\item
An intrinsic structure is one that can be defined and
considered in terms of the submanifold alone, without any refernce to
the embedding: it is the case of the induced conformal, M\"obius and
Laplace structures.
\item An extrinsic kind of structure refers explicitly to (some
  infinitesimal version - like the normal bundle - of) the embedding
  of the submanifold in its conformal (or M\"obius) ambient space: it
  is the case of some tensorial invariants of the embedding and of the
  induced connection on the weightless normal bundle. \end{itemize}

Geometrically, a Laplace structure on a curve is a {\em projective
  structure} \cite{hitch}, \cite{guha} and the global projective
geometry of a closed curve in a conformal (or M\"obius) ambient space
turns out to be a very interesting, and largely unknown problem, as a
forthcoming paper shows \cite{psc}.

After a preliminary section where we recall some basic facts of
conformal geometry (weight bundles, Weyl structures, and curvature
decompositions), with a particular focus on low dimensions, we recall
in Section 3 the definition of the conformal geodesic equation and
define certain tensorial invariants of a conformal embedding.

In the 4th Section, the M\"obius and Laplace structures turn out to
be implicit on a conformal manifold of dimension at least 3, and they
only need to be specified explicitly in low dimensions. On the other
hand, embedded submanifolds in higher-dimensional conformal manifolds
turn out to admit induced M\"obius, resp. Laplace
structures. Relating these induced structures to the implicit
ones (if the submanifold has dimension at least 3) leads again to some
of the tensorial invariants of Section 3.

 Theorem \ref{existt} shows that, besides the particular case
of hypersurfaces, when one of the invariants vanishes identically, as
a consequence of the Gau\ss-Codazzi equations, any given such
tensorial objects on a given manifold and on its normal bundle can be
realized as the invariants of an embedding in some ambient space.

In Section 5, we show that these {\em invariant tensors} of an
embedding turn out to be obstructions for various properties that
generalize, in the conformal context, the {\em totally geodesic
  submanifolds} of Riemannian geometry.

More precisely, a submanifold is called {\em totally umbilic} if it is
totally geodesic for some metric in the conformal class, it is {\em
  weakly geodesic} if it is spanned by conformal geodesics in the
ambient space, and {\em strongly geodesic} if its conformal geodesics
are also conformal geodesics in the ambient space (for dimensions 1 or
2, the conformal structure of the (sub)manifold needs to be completed
(for rigidity) by a Laplace, resp. M\"obius structure).

Finally, these different kinds of geodesic properties of a submanifold
are shown to satisfy some implications (among which the fact that {\em
  strongly geodesic} implies {\em weakly geodesic} turns out to be
non-trivial), and can be characterized by the vanishing of some of the
above mentionned tensorial invariants, Theorem \ref{implic}. 


\section{Preliminaries on conformal geometry}
In this section, we review the main notions needed in conformal
geometry. Good references are \cite{beg} and \cite{ga1}, however we need to push some
of the formulas beyond their usual lower bound for the dimension, like
in \cite{ca} (in
particular for the Schouten-Weyl tensor and the normalized scalar
curvature); a reader familiar with the formalism of weight bundles,
Weyl structures, etc., may jump directly to Proposition \ref{curw}.
\subsection{Weight bundles}
Let $M$ be a $m$-dimensional manifold with density bundle
$|\Lambda|M$. This is an oriented line bundle, hence 
trivial, whose positive sections are the {\em volume elements}
of $M$, allowing the integration of functions on the manifold; it is
isomorphic, if $M$ is oriented, with $\Lambda^m M$, the bundle of
$m$-forms on $M$ (the isomorphism depends on the orientation). 

A {\em conformal structure} on $M$ 
is a {\em positive-definite} symmetric bilinear form $c$ on $TM$ with values
in the line bundle $L^2:=L\otimes L$, or, equivalently, a
non-degenerate section $c\in C^\infty(S^2M\otimes L^2)$ (Here we
denote by $S^2M$ the bundle of symmetric bilinear forms on $TM$), with the
following normalization condition: 
$$|\det c|:(\Lb^mTM)^2\ra (L^2)^m$$
is the identity.
(Note that $(\Lb^mM)^2\simeq (|\Lb|M)^2\simeq L^{2m}$.)



\begin{rem} Each positive section $l$ of $L$ trivializes it, hence
  $$g_l:=l^{-2}c:TM\otimes TM\ra \R$$
is a Riemannian metric on $M$. If $l':=e^fl$ is another positive
section (for $f:M\ra\R$ a smooth function), then the metric
$g_{l'}=e^{-2f}g_l$ is conformally equivalent to $g_l$, and they belong
to the same conformal class, defined by $c$.
\end{rem}

\begin{rem} Because $L$ is a trivial bundle, not only natural powers
  (defined as multiple tensor products: $L^k=L\ot...\ot L$) or
  negative integer powers ($L^{-1}$ is the dual of $L$ and
  $L^{-k}:=L^*\ot...\ot L^*$) are well-defined, but also {\em real}
  powers of $L$: indeed, the bundle $L^k$, $k\in\R$, is the  bundle associated to the frame
  bundle $\mathrm{GL}(M)$ and the representation $|\det|^{\frac k n}$.\end{rem}
\begin{rem} A
  conformal structure is equivalent to a reduction of $\GL(m)$ to the
  $\CO(m)=\OO(m)\times\R^*_+$--bundle of conformal frames $\CO(M)$.\end{rem}
{\noindent\bf Convention. } The usual
identifications of vectors and covectors from Riemannian
geometry can be applied in the conformal setting, but it involves a
tensor product with a corresponding weight bundle: $T^*M\simeq TM\ot
L^{-2}$. In general, an irreducible representation of $\CO(m)$ is the
tensor product of an irreducible representation of $\OO(m)$  and one of
$\R^*_+$, the latter being the multiplication by the $k$th power of an
element of $\R_+^*$, where $k\in\R$ is called the {\em conformal
  weight} of the representation. In particular, for the associated
bundles, $TM$ has conformal weight $1$, $T^*M$ has conformal weight
$-1$, a $k$-weighted 
$(r,s)$ tensor bundle $A\subset \otimes^r T^*M\otimes \otimes^s TM\ot L^k$ has weight
$s-r+k$ (in particular, any endomorphism bundle has conformal weight zero).
Two irreducible bundles are isomorphic as $\CO(M)$--bundles if and only
if they are isomorphic as $\OO(M)$--bundles (for any metric in the
conformal class) {\em and} they have the same conformal weight.

\subsection{Weyl structures}
Unlike in (semi-) Riemannian geometry, a conformal manifold does not
carry a canonical affine connection. Instead, there is a family of adapted
connections, the {\em Weyl structures}:

\begin{defi} A Weyl structure $\nb$ on a conformal manifold $(M,c)$ is a
  torsion-free, {\em conformal} connection on $TM$, {\em i.e.} $\nb
  c=0$.
\end{defi}
The fundamental theorem of conformal geometry can now be stated:
\begin{thrm}\label{fund}\cite{weyl} Let $(M,c)$ be a conformal manifold and
  denote, for a Weyl structure $\nb$, by $\nb^L$ the connection
  induced by $\nb$ on $L$. The correspondence 
$$\left\{\mbox{Weyl structures on $M$}\right\}\lra\left\{\mbox{connections
  on $L$}\right\},$$
given by $\nb\mapsto \nb^L$ is one-to-one.
\end{thrm}
More precisely, the inverse map is given by the following {\em
  conformal Koszul formula} \cite{ca}, \cite{ga1}, \cite{weyl}:
\be\label{kosz}\begin{split}
2c\left(\nb_XY,Z\right) =&
\nb^L_X\left(c(Y,Z)\right)+
\nb^L_Y\left(c(X,Z)\right)-
\nb^L_Z\left(c(X,Y)\right)+\\  
&c\left([X,Y],Z\right)+ 
c\left([X,Z],Y\right)-
c\left([Y,Z],X\right).
\end{split}\ee 



A consequence of Theorem \ref{fund} and
(\ref{kosz}) is the relation between two Weyl structures: as the
difference between two linear connections ${\nb'}^L$ and $\nb^L$  on the line
bundle $L$ is a $1$--form $\te$, the difference between the
corresponding Weyl structures $\nb'$ and $\nb$ must be given by a tensor
$\tilde\te$ that depends linearly on $\te$. More precisely, applying (\ref{kosz})
we get:

\be\label{e1}\nb'_XY-\nb_XY=\tilde\te_XY:=(\te\wedge X)(Y)+\te(X)Y,\ee
where $\te\wedge X$, the wedge product of a $1$--form and a vector, is
the skew-symmetric endomorphism of $TM$ defined by
\be\label{e2}(\te\wedge X)(Y):=\te(Y)X-c(X,Y)\te .\ee
Here, note that the $1$--form $\te$ is a section of $T^*M\simeq
TM\otimes L^{-2}$ and thus $c(X,Y)\te $ is a section of $TM\otimes
L^{-2}\otimes L^2\simeq TM$.


\begin{rem}\label{rqtl} The difference tensor $\tilde\te_X$ from
  \eqref{e1} lies in the adjoint bundle $\mathfrak{co}(M)$ of the
  bundle of conformal frames $\CO(M)$, and thus the difference of the
  induced connections by $\nb,\nb'$ satisfying \eqref{e1} on some
  (weighted) tensor bundle $E$ of conformal weight $k\in\R$ is given by
$$\nb'_X\xi-\nb_X\xi=\tilde\te_X\xi=(\te\we X)\xi+k\te(X)\xi,$$
where $(\te\we X)\xi$ is the usual action of skew-symmetric
endomorphisms on a tensor $\xi$, 
and $k\te(X)\xi$ is the Lie algebra action of $\te(X)\idt$ on a
representation of conformal weight $k$.
\end{rem}
For example, for a section $l$ of $L^k$ we have
\beq\label{der}\nb'_Xl-\nb_Xl=k\theta(X)l,
\eeq
and, for an endomorphism $A:TM\ra TM$, we have
\be\label{end}\nb'_XA-\nb_XA=[\te\wedge X,A],\ee
where the square bracket is the commutator of endomorphisms.
.

\subsection{Curvature}
The curvature of a Weyl structure $\nb$ is defined by 
$$R^\nb_{X,Y}Z:=\nb_X(\nb_YZ)-\nb_Y(\nb_XZ)-\nb_{[X,Y]}Z,$$
and can be seen as a $2$--form with values in $\mathfrak{co}(M)$. The
identity component of this 2--form is the {\em Faraday}
form $F^\nb\in\Lambda^2M$, which is the curvature of the connection
$\nb^L$, more precisely
$$R^\nb_{X,Y}=(R^\nb_{X,Y})^{skew}+F^\nb \otimes \idt.$$
$R^\nb$ satisfies the Bianchi identities (tensorial and differential), 
the tensorial (or the {\em first}) Bianchi identity being: 
\be\label{bi1}R^\nb_{X,Y}Z+R^\nb_{Y,Z}X+R^\nb_{Z,X}Y=0,\ \forall
X,Y,Z\in TM.\ee 
\begin{defi} The {\em suspension} of a bilinear form
$A\in T^*M\otimes T^*M$ by the identity is a $2$--form $A\we\idt$ with
values in $\so(M)$, given by
\be\label{susp}(A\wedge
\idt)_{X,Y}:=A(Y,\cdot)\wedge X-A(X,\cdot)\wedge Y,\ee
where we note, as in (\ref{e2}), that the wedge product between a
vector and a 1--form (here $A(Y,\cdot)$) is a skew-symmetric
endomorphism.\end{defi} 
The tensorial Bianchi identity \eqref{bi1} is satisfied by the tensor 
\be\label{tF}\tilde F^\nb:=-\frac 12 F^\nb\wedge \idt + F^\nb\otimes \idt.\ee
If we define the {\em Riemannian component} of $R^\nb$ as
$$R^{\nb,Riem}_{X,Y}:=(R^\nb_{X,Y})^{skew}+\frac12F^\nb\wedge\idt,$$
then it is a $2$-form with values in $\mathfrak{so}(M)$ that satisfies
the first Bianchi identity (like a Riemannian curvature tensor).

\begin{rem}\label{kersus} The linear map
$$\End(TM)\ni A\longmapsto A\we\idt\in\Lb^2M\ot\so(M)$$
is injective for $n\ge 3$, and its kernel consists of all trace-free
endomorphisms if $n=2$. On the other hand, every tensor in
$\Lb^2M\ot\co(M)$ satisfies the Bianchi identity for $n=2$, hence the
correction term $-\frac{1}{2}F^\nb\we\idt$ (which  is actually zero) is not
needed there.\end{rem}


In general, the {\em Ricci contraction} of a curvature tensor is the
map
$$\ric:\Lb^2M\ot\co(M)\ra\End(TM),$$
$$\ric(R)(X,Y):=\tr\left(R_{\cdot,X}Y\right),$$
and the {\em Ricci tensor} of $\nb$ is $\ric^\nb:=\ric(R^\nb)$. 

A straightforward computation shows
\be\label{cric} \ric(A\we\idt)=(n-2)A+\tr_c A\cdot c.\ee
This implies
$$\ric(\tl F^\nb)=-\frac{n-2}2F^\nb-F^\nb=-\frac{n}2F^\nb.$$
On the other hand, the Ricci contraction, applied to the Riemannian
component $R^{\nb,Riem}$ of $R^\nb$, produces a symmetric tensor 
$$\ric^{\nb,s}:=\ric(R^{\nb,Riem})\in S^2M.$$
Therefore, the Ricci tensor $\ric^\nb:=\ric(R^\nb)$ of a Weyl
structure $\nb$ has a skew-symmetric part, equal to $-n/2F^\nb$, and a
symmetric part, equal to
$\ric^{\nb,s}=\ric(R^{\nb,Riem})$. 

THe relation (\ref{cric}) also implies that, if $n\ge 3$, then, for every given
bilinear form $A\in \otimes^2M$, there exists another bilinear form
$h(A)\in \otimes^2M$ such that $\ric(h(A)\wedge \idt)=A$. 

Straightforward computations show that the linear map
$h:\otimes^2M\ra\otimes^2M$ has the expression
\be\label{h}h(A)=\frac1{n-2}A^s_0+\frac1{2n(n-1)}\tr_cA\cdot
c-\frac12A^{skew},\ee
where $A^s_0\in S^2_0M$ is the symmetric, trace-free part of $A$, $\tr
A\in L^{-2}$ is the trace of $A$ w.r.t. $c$, and $A^{skew}\in \Lb^2(M)$
is the skew-symmetric part of $A$. 
\begin{defi} 
The {\em scalar curvature} of an Weyl structure $\nb$ on the conformal
manifold $(M,c)$ is the density of weight $-2$
\be\label{defsc}\scal^\nb:=\tr_c(\ric^\nb),\ee 
given by the trace (with respect to $c$)
of the Ricci tensor of $\nb$.

If $n:=\dim M\ge 3$, the {\em Schouten-Weyl tensor} $h^\nb$ of a Weyl
  structure $\nb$ on $(M,c)$ is the tensor
  \be\label{defsch}h^\nb:=h(\ric^\nb)\in\otimes^2M.\ee
 Its symmetric part
  $h^{\nb,s}:=h(\ric^{\nb,s})$ is called the symmetric Schouten-Weyl
  tensor of $\nb$. 
\end{defi}
In particular, for $\nb$ the Levi-Civita connection of the metric $g$,
we have the well-known formula
\be\label{RSg}h^g=\frac1{n-2}\ric^g_0+\frac1{2n(n-1)}\scal^g\cdot
g.\ee
\begin{rem}\label{huniq} From Remark \ref{kersus} it follows that, even if the
  Schouten-Weyl tensor on $\nb$ is only defined if $n\ge 3$, there
  always exist bilinear forms $h\in\ot^2M$ such that $\ric(h\we
  \idt)=\Ric^\nb$: for $n\ge 3$ this requires $h$ to be the Schouten-Weyl tensor
  $h^\nb$ as defined in \eqref{defsch}; for $n\ge 2$ the skew-symmetric part of $h$ has to be
  $-\frac12F^\nb$ and the pure trace part has to be
  $\frac14\scal^\nb\cdot c\in S^2M$
  (the trace-free part of $h$ is undetermined if $n=2$); and for $n=1$ $h$ is
  an undetermined function.  We can therefore define 
\be\label{defsig}\sigma^\nb:=\frac{1}{2(n-1)}\scal^\nb\in L^{-2}\ee
as the {\em normalized scalar curvature}, for $n\ge 2$, and $\frac1n\sigma^\nb\cdot c$ is the pure trace part of any
of the tensors $h$ above (in particular, if $n\ge 3$, $\sigma^\nb$ is
the
trace of the Schouten-Weyl tensor  $h^\nb$) .\end{rem} 
The Schouten-Weyl tensor is sometimes called the normalized Ricci
tensor; it is equal to $k$ times the metric on every Riemannian
manifold of constant sectional curvature $k$.


After substracting from $R^{\nb,Riem}$ the suspension of $h^{\nb,s}$,
we obtain a tensor $W^\nb\in\Lb^2M\ot\so(M)$ that satisfies the Bianchi
identity, and is trace-free (up to a constant, the only non-trivial
contraction (or trace) on the space of the Riemannian curvature
tensors is the Ricci contraction). This tensor is called the {\em Weyl
  tensor} of $\nb$.
 
The following decomposition is a direct consequence of \eqref{cric}
and of the definition above:
\begin{prop}\label{curw} The curvature $R^\nb$ of a Weyl structure
  $\nb$ on a conformal manifold $(M,c)$ of dimension $n\ge 3$ decomposes as
\be\label{curvdec}R^\nb=h^\nb\we \idt + W+F^\nb\ot \idt,\ee
where 
$W^\nb$ is the Weyl tensor, $F^\nb$ is the Faraday $2$-form, and
$$h^\nb=h^{\nb,s}-\frac{1}{2}F^\nb$$
 is the full Schouten-Weyl tensor of $\nb$.
\end{prop}
Note that $W=0$ if $n=3$ by dimension reasons. 

If $n=2$, the curvature decomposition is
even simpler, since $\Lb^2M\ot\co(M)$ is 2-dimensional (and
automatically satisfies the Bianchi identity):
$$R^\nb=\frac12\s^\nb\cdot c\we\idt+F^\nb\ot\idt,$$
where $\s^\nb=\frac12\scal^\nb$ is a section of $L^{-2}$.

We conclude that the curvature
tensor $R^\nb$ of a conformal manifold $(M,c)$ of dimension $n$ is determined by\bi
\item its Faraday curvature $F^\nb\in\Lb^2M$ (for $n\ge 2$)
\item its scalar curvature $\scal^\nb:=tr_c\ric^\nb\in L^{-2}$, or,
  equivalently, its normalized scalar curvature
  $\s^\nb=\frac1{2(n-1)}\scal^\nb$ (for $n\ge 2$)
\item its trace-free symmetric Ricci tensor $\ric^{\nb,s}_0\in S^2_0M$
  (for $n\ge 3$)
\item its trace-free part $W^\nb$, the Weyl tensor (for $n\ge 4$).\ei
All these components are sections in $\CO(n)$-irreducible vector
bundles. 


We give now the transformation rule for the curvature tensors
corresponding to two Weyl structures, more precisely, the
transformation rules for their $\CO(n)$ irreducible components in the
list above:


\begin{prop}\label{curch}
For two Weyl structures $\nb'=\nb+\tilde\theta$ on a conformal
manifold $(M,c)$ of dimension $n\ge 3$, the corresponding
Schouten-Weyl tensors are related by:
\be\label{scht}h^{\nb'}-h^\nb=-\nb\theta
+\theta\otimes\theta-\frac{1}{2}c(\theta,\theta)c.\ee 
Moreover, the Weyl tensor $W$ is independent on the Weyl structure and
depends on the conformal structure only. The Faraday curvature changes
as follows:
$$F^{\nb'}=F^\nb+d\te.$$
\end{prop}
\begin{proof}
Let us compute the curvature of $\nb'$ by deriving (\ref{e1}), and
using (\ref{end}) at a point where the $\nb$--derivatives of the
involved vector fields vanish:
\be\label{11}\nb'_X\nb'_YZ=\nb_X\nb_YZ+(\nb_X\theta\wedge
Y)(Z)+(\tl\theta_X)\circ(\tl\theta_Y)(Z)+(\nb_X\theta)(Y)Z,\ee 
Therefore
$$R^{\nb'}_{X,Y}Z-R^\nb_{X,Y}Z=(\nb_X\te\we Y)(Z)-(\nb\te_Y\te\we
X)(Z)+d\te(X,Y)Z+[\tl\te_X,\tl\te_Y](Z).$$ 
Note that $[\tl\te_X,\tl\te_Y]=[\theta\we X,\theta\we Y]$, hence we get
$$R^{\nb'}_{X,Y}Z=R^\nb_{X,Y}Z-(\nb\theta\wedge \idt)_{X,Y}Z+[\theta\wedge
  X,\theta\wedge Y](Z)+d\theta(X,Y)Z.$$
We compute directly
\bea[\theta\wedge X,\theta\wedge Y]&=&(\theta\ot\theta)(Y)\wedge
X-(\theta\ot\theta)(X)\wedge Y+c(\theta,\theta)X\wedge Y\\
&=&\left(((\theta\ot\theta)\wedge
\idt)_{X,Y}-\frac{1}{2}c(\theta,\theta)(c\wedge \idt)_{X,Y}\right),\eea
that implies
$$F^{\nb'}=F^\nb+d\theta,\quad W^{\nb'}=W^\nb$$
and the result claimed in \eqref{scht}.
\end{proof}
\begin{cor}\label{scalcf} Let $\nb,\nb'$ be Weyl structures on the
  conformal manifold $(M,c)$ of dimension at least $3$, such that
  \eqref{e1} holds. The relations between the trace-free parts,
  resp. the traces of the Schouten-Weyl tensors of $\nb$ and $ \nb'$ are:
\beay h^{\nb'}_0&=&h^\nb_0-(\nb\te)_0+(\te\ot\te)_0,\label{h0}\\
\s^{\nb'}&=&\s^\nb+\dl^\nb\te+\frac{2-n}{2}c(\te,\te).\label{s0}\eeay
Here $\dl^\nb\te:=-\tr_c\nb\te=-\sum_{i=1}^n(\nb_{e_i}\te)(e_i)$, for $\{e_i\}$ an
$c$-orthonormal basis of $TM$.
\end{cor}
\begin{rem}\label{RS} 
Proposition \ref{curch} holds regardless of the dimension $n$ of $M$
in the following sense: {\em assuming} that the curvature tensors $R^\nb$ has the expression
(\ref{curvdec}), for {\em some} tensor $h^\nb$ (of  Schouten-Weyl type), then
$R^{\nb'}$ has also an expression (\ref{curvdec}), with $h^{\nb'}$
given by \eqref{scht}. As mentioned in Remark \ref{huniq}, if $n=2$ only the
skew-symmetric part, and the pure trace part of $h^\nb$ are determined
by the Ricci tensor $\ric^\nb$. The transformation rule \eqref{s0} is,
thus, also valid for $n=2$.

Note, however, that even if the curvature of a Weyl structure does neither
define a tensor of type $h^\nb_0$ in $S^2_0M$ for $n= 2$ nor a density
of type $\s^\nb$ in $L^{-2}$ for
$n=1$, these bundles are not zero themselves. We will see that some
extra geometric structures on $(M,c)$ induce, for all Weyl structures
$\nb$, sections $h^\nb_0$, resp. $\s^\nb$ in these bundles, such that
the transformation rules from Corollary \ref{scalcf} hold.

Indeed, following D. Calderbank \cite{ca}, we introduce the 
{\em M\"obius differential operator} on a conformal manifold:
\end{rem}





\section{M\"obius and Laplace structures on conformal manifolds}


\begin{prop}\label{mobex}\cite{beg}, \cite{bc}, \cite{ca}, \cite{ga}. Let $(M,c)$ be
  a conformal manifold of dimension $m\ge 3$, and let $\hs_0^\nb$,
  $h_0^\nb$ be trace-free Hessian, resp. trace-free Schouten tensor of
  a Weyl structure $\nb$. Then the second order differential operator
  $\M^\nb:C^\infty(L)\ra C^\infty(S^2_0M\otimes L)$  defined by
\beq\label{cm}\M^\nb_{(X,Y)}l:=\hs^\nb_0(X,Y)l+h^{\nb,s}_0(X,Y)l,\eeq
is independent of $\nb$. Here $h^{\nb,s}_0$ is the symmetric
trace-free Schouten-Weyl tensor of $\nb$.
\end{prop}
\begin{proof}
Using \eqref{der} we easily get 
\beq\label{hess}\begin{split}
\hs^{\nb'}(X,Y)l-&\hs^\nb(X,Y)l=(k-1)[\theta(X)\nb_Yl+\theta(Y)\nb_Xl]
+c(X,Y)\nb\theta l\\
& +k[(\nb_X\theta)(Y)+(k-2)\theta(X)\theta(Y)+c(X,Y)c(\theta,\theta)]l.
\end{split}\eeq
Taking $k=1$ and using \eqref{scht} shows that
$\mbox{Hess}^\nb_0(X,Y)l+h^{\nb,s}_0(X,Y)l$ does not depend on the choice
of the Weyl structure $\nb$. 
\end{proof}
As an immediate corollary, we see that the difference of the
trace-free Hessians associated to two arbitrary Weyl structures is a
scalar operator on $L$. This motivates the following
\begin{defi}\label{mob}{\bf (Calderbank \cite{ca})} A {\em M\"obius structure} on a 
conformal manifold $(M,c)$ of dimension at least $2$ is a second order linear
  differential operator 
$$\M:C^\infty(L)\ra C^\infty(S^2_0M\otimes L)$$
such that $\M-\hs^\nb_0$ is a scalar operator 
for some (and thus all) Weyl structures $\nb$ on $(M,c)$.
\end{defi}
The M\"obius structure defined by \eqref{cm} is called 
the {\em canonical M\"obius structure} of $(M,c)$ and will be denoted
by $\M^c$. 
\begin{rem}\label{moeb3} If $\dim M\ge 3$, then every M\"obius
  structure $\M$ is the sum of the canonical M\"obius structure $\M^c$
  and a trace-free symmetric bilinear form on $M$. If $\dim M=2$, we {\em
    define} $h^{\M,\nb}_0:=\M-\hs^\nb_0$ as a symmetric bilinear form
  and we conclude from \eqref{hess} that $h^{\M,\nb}_0$ satisfies the
  transformation rule \eqref{h0}. In fact, we have:\end{rem}
\begin{prop}\label{moeb-h}\cite{ca} A M\"obius structure $\M$ on a conformal
  manifold $(M,c)$ is equivalent to a map 
$$h^\M_0:\{\mbox{Weyl structures on $(M,c)$}\}\ra C^\infty (S^2_0M)$$
such that if $\nb,\nb'$ are Weyl structures satisfying \eqref{e1}, then
$h^{\M,\nb}_0:=h^\M_0(\nb)$ and $h^{\M,\nb'}_0:=h^\M_0(\nb')$ satisfy
\eqref{h0}.
\end{prop}
\begin{rem}\label{schtm} On a M\"obius surface, the trace-part of the
  Schouten-Weyl tensor, $\sigma^\nb$, is well-defined for every Weyl
  structure $\nb$ (it uses just the underlying conformal structure, see
  Remark \ref{RS}), and the same holds for the Faraday form (hence for
  the skew-symmetric part of what should be the Schouten-Weyl
  tensor). The M\"obius structure, in turn, associates to 
  $\nb$ the symmetric, trace-free {\em M\"obius Schouten-Weyl tensor}
  $h^{\M,\nb}_0$ as above. This means that, on a M\"obius surface,
  each Weyl structure has its own Schouten-Weyl tensor, just like in
  the case of a conformal manifold of higher dimension.\end{rem}

A similar construction (which yields the well-known conformal Laplacian) 
holds for the trace of the Hessian of Weyl structures.  
Note that the Schouten-Weyl tensor is only defined for conformal manifolds
of dimension $m>2$, but, in dimension 2, its trace
$\sigma^\nb$ (which is a multiple of the scalar curvature) is still
well-defined, and we have the following well-known fact: 

\begin{prop}\label{lapex} On a conformal manifold of dimension $m\ge 2$, the 
second order differential operator $\Lp^\nb:C^\infty(L^k)\ra C^\infty(L^{k-2})$ 
$$\Lp l=\Lp^c l:=\tr \left(\hs^\nb l\right)
+\left(1-\frac{m}{2}\right)\sigma^\nb l,$$ where $\nb$ is any Weyl
structure and $\hs^\nb$, $\sigma^\nb$ are the corresponding Hessian,
resp. pure-trace part of Schouten tensor of $\nb$, is independent on
$\nb$ for $k=1-\frac m2$.
\end{prop}
\begin{proof} This follows directly after taking the conformal trace
  in \eqref{scht} and \eqref{hess}.
\end{proof}
The operator $\Lp^c$ is the {\em conformal Laplacian} of the conformal
manifold $(M,c)$, or the {\em Yamabe operator}.

Proposition \ref{lapex} motivates the following
\begin{defi}\label{lap} A {\em Laplace} structure on an
$m$--dimensional conformal
manifold $(M,c)$ is a second order linear differential operator
$$\Lp:C^\infty(L^k)\ra C^\infty(L^{k-2}),$$
where $k:=1-m/2$, such that $\Lp-\tr \left(\hs^\nb \right)$ is a
scalar operator for some (and thus all) Weyl structures $\nb$ on
$(M,c)$. 
\end{defi}

The Laplace structure defined by Proposition \ref{lapex} is called 
the {\em canonical Laplace structure} of $(M,c)$. 

\begin{rem}\label{lap2} If $\dim M\ge 2$, then every Laplace
  structure $\Lp$ is the sum of the canonical Laplace structure $\Lp^c$
  and a section of $L^{-2}$. If $\dim M=1$ and we have fixed a Laplace
  structure $\Lp$ on the curve $M$, we {\em
    define} $\s^{\M,\nb}:=\Lp-\hs^\nb$ as a section of $L^{-2}$
  and we conclude from \eqref{hess} that $\s^{\M,\nb}$ satisfies the
  transformation rule \eqref{s0}. In fact, as in Proposition
  \ref{moeb-h}, we immediately get:\end{rem}
\begin{prop}\label{laplace}\cite{bc} A Laplace structure $\Lp$ on a conformal
  manifold $(M,c)$ is equivalent to a map 
$$\s^\Lp:\{\mbox{Weyl structures on $(M,c)$}\}\ra C^\infty (L^{-2})$$
such that if $\nb,\nb'$ are Weyl structures satisfying \eqref{e1}, then
$\s^{\Lp,\nb}:=\s^\Lp(\nb)$ and $\s^{\Lp,\nb'}:=\s^\Lp(\nb')$ satisfy
\eqref{s0}.
\end{prop}
\begin{rem}\label{schtl} Like in the case of a M\"obius surface
  (Remark \ref{schtm}), on a Laplace curve $(C,\Lp)$,
  $\sigma^{\Lp,\nb}$ plays therefore the role of the Schouten-Weyl
  tensor associated to the connection $\nb$. 
In \cite{bc}, the authors use also the name {\em M\"obius structures}
for 1-dimensional Laplace structures. 
\end{rem}

It turns out that M\"obius and Laplace structures, although
implicitely present in
higher dimensions as well, are particularly important in dimensions 2
and 1 respectively, because they provide to an otherwise {\em flexible}
conformal structure is such a low dimension the rigidity that is
implicit in higher dimensions.




\begin{rem} Obviously, any M\"obius, resp. Laplace structure
  determine, through their principal symbols, the underlying conformal
  structure. The two propositions above show that, for higher
  dimensions, the converse also holds (but in a non-trivial way). In
  the next section we will see that any submanifold in a conformal
  $m$-dimensional manifold $(M,c)$ inherits an {\em induced} M\"obius
  structure (for $m\ge 3$) and an {\em induced} Laplace structure (if
  $m\ge 3$ or if $m=2$ and $M$ has, additionally, a M\"obius structure). 
\end{rem}
\begin{rem} The trace-free, resp. trace of \eqref{hess} shows that,
  unless the weight $k$ is equal to $1$, resp. to $\frac{2-m}{2}$, the
  trace-free Hessian, resp. the trace of the Hessian can not be corrected
  merely by a scalar term in order to produce a conformally invariant
  operator (in general, a first-order correction term is also needed). This implies that the M\"obius and the Laplace operators, as
  defined in Propositions \ref{mobex} and \ref{lapex} (including the
  specified conformal weights) are the only second-order, linear
  differential operators on weight bundles, that can be canonically
  associated to a conformal structure.

\subsection{Geometric meaning}
M\"obius and Laplace structures have the following geometric interpretation: a non-vanishing section $l$ of the
corresponding weight bundle satisfies $\mathcal{P}l=0$ if and only if
the correction term in the defining formula from Proposition
\ref{mobex}, resp. \ref{lapex}, when using the Levi-Civita connection $\nb^l$ of the
Riemannian metric $g^l$ associated to $l$, vanishes identically, i.e.
$$\M=\hs_0^{\nb^l},\mbox{ resp. }\Lp=\tr(\hs^{\nb^l}).$$
This in turn means that $\M l=0$ (for $n>2$) if and only if the metric $g^l$ is
Einstein, and $\Lp l=0$ (for $n>1$) if and only if $g^l$ is
scalar-flat, which is a solution to the $s=0$ {\em Yamabe problem}: finding
metrics with (constant) scalar curvature equal to $s$ (here $s=0$) in
the conformal class. The general Yamabe problem (finding metrics with
constant, non-zero, scalar curvature in the conformal class)
corresponds to the 
eigenvectors of $\Lp$ (in order to identify $L^k$ with $L^{k-2}$, so
that the concept of eigenvalue makes sense, the solution $l$ itself is
used, which turns the general Yamabe problem into a non-linear one). 
\end{rem}
A little weaker than to find solutions to $\M l=0$ is to look for Weyl
structures, such that their trace-free Hessian is equal to $\M$
(therefore, the trace-free symmetric Schouten-Weyl tensor
(or, equivalently, Ricci tensor) vanishes); these connections are
called {\em Einstein-Weyl structures}. For a M\"obius surface, {\em
  M\"obius-Einstein-Weyl structures} can be defined as Weyl structures
$\nb$ such that $\M=\hs^\nb_0$ (in \cite{ca}, the author defines
2-dimensional Einstein-Weyl structures on a conformal surface as
above, but requiring in addition that $\M$ -- which is not fixed in
this case -- is some {\em flat} M\"obius structure, i.e., it is
locally given by a conformal chart).

Analoguously, the question of finding parametrizations
(note that a metric on a curve is equivalent to a parametrization) of
a given Laplace curve $(C,\Lp)$ such that the zero-order term of 
$\Lp$, in this parametrization, is constant, can be seen as an
extension to Laplace (or, equivalently, projective) curves of the
Yamabe problem \cite{psc}.

\section{Extrinsic conformal geometry}
In this section we define three conformally invariant tensors that
characterize the embedding of two conformal manifolds, just like the
second fundamental form does in the Riemannian setting. The goal is to
define various geometric notions of {\em totally geodesic} conformal
embeddings (using the {\em conformal geodesics}, to be defined below)
and characterize them by the vanishing of (some of) the above mentioned invariants.

\subsection{Conformal geodesics}
Let $M^m,$ be a smooth manifold, endowed with a conformal structure if
$m\ge 3$, a M\"obius structure if $m=2$ or a Laplace structure if $m=1$.
There is a way to define a {\em conformal acceleration} of an (immersed) parametrized curve in $M$:
\begin{defi} Let $\c:I\to C\subset M$ be an immersion. Choose any Riemannian metric $g\in c$ and denote by $h^g$
  its Schouten tensor defined by \eqref{RSg} for $m>2$, and by
  Proposition \ref{moeb-h} for $m=2$.
Then the following vector field along the curve $\c$
\beq\label{acc}
a(\g):=g(\dg,\dg) (h^g(\dg))^\sharp-\dddg+3\frac{g(\dg,\ddg)}{g(\dg,\dg)}\ddg+
\bigg(-6\frac{g(\dg,\ddg)^2}{g(\dg,\dg)^2}+\frac32\frac{g(\ddg,\ddg)}{g(\dg,\dg)}+
2\frac{g(\dg,\dddg)}{g(\dg,\dg)}\bigg)\dg 
\eeq
is called the
{\em conformal acceleration} of $\g$. The notation $(h^g(\dg))^\sharp$ stands here for the vector field associated by the metric $g$ to the 1-form $h^g(\dg)$ along $\gamma$ and $\ddg:=\nabla^g_{\dg}\dg$,  $\dddg:=\nabla^g_{\dg}\ddg$.
\end{defi}

It is well known  (cf. \cite[p. 67]{ga}, \cite[p. 218]{baea} or Lemma \ref{accel} below) that the conformal acceleration of a parametrized curve does not depend on the metric $g$.

The following observation is folklore (cf. e.g. \cite{baea}):

\begin{lem} The normal part of $a(\g)$ does not depend on the
parametrization $\g$ of $C$. \end{lem}

{\em Conformal geodesics} on a conformal manifold (resp. a
M\"obius surface or a Laplace curve) are then defined, by analogy with
the notion of a geodesic of a connection, by the condition that the conformal
acceleration vanishes \cite{ferr}, \cite{phys}, \cite{ga}:

\begin{defi}\label{cgeod}
A regular curve $\c:I\ra M$ in a conformal manifold
$(M^m,c)$, with $m\ge 3$, a M\"obius surface for $m=2$, or a Laplace
curve for $m=1$, is called {\em parametrized 
  conformal geodesic} if its conformal acceleration
$a(\c)$ vanishes identically and {\em
unparametrized conformal geodesic} if the normal part of the
conformal acceleration vanishes.\end{defi}
The Cauchy-Lipschitz Theorem easily shows that every
unparametrized conformal geodesic has local parametrizations turning
it into a parametrized conformal geodesic. 

\begin{rem} The notion of conformal geodesic is also known in the
  literature to emerge out of the canonical {\em Cartan connection}
  associated to a given conformal $n$-manifold $(n\ge 3)$, resp. a
  M\"obius surface: The Cartan connection is a total parallelism 
$$\omega: \T P\ra \mathfrak{g},$$ on a
  certain principal bundle $P$ over $M$, with values in the Lie
  algebra of the M\"obius group $G:=\SO(n+1,1)$. The integral curves of
  the fundamental vector fields $X^A$ on $P$, defined by elements
  $A\in\mathfrak{g}$, are parametrized curves on $P$ that project on
  parametrized conformal geodesics on $M$ in the sense of the Definition \ref{cgeod}, see
  \cite{cap}, \cite{parab}.
\end{rem}

\begin{rem} E. Musso \cite{mus} uses the terminology {\em ``conformal geodesics''} to
  denote some curves that are extremal with respect to a functional
  which consists of integrating the {\em
    conformal arc length} \cite{ssu}, \cite{su} (which, in our terms,
  is the square root of the norm of the normal component of $a(\c)$)
  along the curve $\c$, assuming that this quantity is nowhere
  vanishing. Note that for our definition of the conformal geodesics,
  this ``conformal arc length'' vanishes identically, therefore Musso's
  notion of ``conformal geodesics'' never coincides with the one we use
  here. In fact, we rather consider Musso's ``conformal arc length''
  to be a conformal analogon of the {\em (extrinsic) curvature} of the
  curve; the {\em arc length parametrizations} of a curve in
  a Riemannian manifold is rather an {\em intrinsic} structure on the
  curve, as is the induced {\em Laplace} (or projective) structure on
  a curve in a conformal manifold / M\"obius surface (see below and,
  for more details, \cite{psc}).
\end{rem}

Equivalently, the conformal geodesics can be defined as follows
(see also \cite{ferr}, \cite{ga}): first, Corollary \ref{adapt} below
shows that every parametrized curve is geodesic for a certain Weyl
structure which is then called {\em
  adapted}; then, the conformal, or M\"obius,
resp. Laplace geodesics can be defined as follows: 
\begin{defi}\label{cmg} Let $(M,c)$ be a conformal manifold of
  dimension at least  3, a M\"obius surface or a Laplace curve. A regular curve
  $\gamma:I\ra M$ is a parametrized {\em conformal}, resp. {\em
    M\"obius geodesic} if any adapted Weyl structure $\nb$ (i.e., for
  which $\nb_{\dg}\dg=0$)  satifies
\be\label{cg}h^\nb(\dg)=0\ee
on $I$. It is an unparametrized conformal geodesic if the section of $\nu$
induced by $h^\nb(\dg)$ vanishes identically. \end{defi}
\begin{rem} The transformation rule \eqref{scht} shows that the
  quantity $h^\nb(\dg)$ depends only on the restriction of the Weyl
  structure $\nb$ along the curve $\gamma(I)$. Note that no other
  similar tensor defined in the previous section ($h^{\nb,s}$,
  $\sigma^\nb$, $\ric^\nb$, etc.) has this property.\end{rem}  
The equivalence of this definition with the previous one is a
consequence of
\begin{lem}\label{accel} The conformal acceleration of a parametrized curve $\c$ is
  equal to the vector field $c(\dg,\dg)h^\nb(\dg)$ along $\c$,
for any Weyl structure $\nabla$ on $(M,c)$, adapted to $\g$.
\end{lem}
\begin{proof} Choose any metric $g\in c$, and write $c=l^2\,g$ for some section $l$ of $L$. We denote by $\nabla^g$ the Levi-Civita connection of $g$ and by $f^2:=g(\dg,\dg)$. Then 
\be\label{q1}c(\dg,\dg)=f^2l^2,\ee
and after two covariant derivatives with respect to $\dg$:
\be\label{q2}g(\ddg,\dg)=ff'\qquad \hbox{and}\qquad g(\ddg,\ddg)+g(\dg,\dddg)=f'^2+ff'',\ee
where we denoted by $f':=\dg(f)$ and $f'':=\dg(f')$.

Let $\nabla:=\nabla^g+\tilde\theta$ be a Weyl structure on $(M,c)$ and
let $T:=l^2\theta$ be the dual vector field to $\theta$ by $g$. 
By \eqref{e1} and \eqref{q1}, $\nabla$ is adapted to $\gamma$ if and only if 
$$0=\nabla^g_{\dg}\dg+2\theta(\dg)\dg-c(\dg,\dg)\theta=\nabla^g_{\dg}\dg+2l^{-2}g(T,\dg)\dg-f^2T.$$
Taking the scalar product with $\dg$ and using \eqref{q2} we obtain $g(T,\dg)=-\tfrac{f'}f$. Reinjecting this in the same equation yields
\be\label{q3}T=\frac{\ddg}{f^2}-2\frac{f'}{f^3}\dg.\ee
Incidentally, this just proves the uniqueness of $\nabla$ along
$\gamma$, fact that we will prove in Proposition \ref{transcx} to hold in a
more general setting. From \eqref{q2} and \eqref{q2} we obtain
\be\label{q4}g(T,T)=\frac{g(\ddg,\ddg)}{f^4}.\ee

On the other hand, Equation \eqref{scht} (with $\nabla'$ and $\nabla$ there replaced by $\nabla$ and $\nabla^g$ respectively) gives
$$h^\nabla=h^g-\nabla^g\theta+\theta\otimes\theta-\tfrac12c(\theta,\theta)c.$$
Plugging $\dg$ in this formula and using \eqref{q1} and \eqref{q3} yields:
\bea c(\dg,\dg)h^\nabla(\dg)&=&c(\dg,\dg)\left(h^g(\dg)-\nabla^g_{\dg}\theta-\frac{f'}{f}\theta-\frac12l^{-2}g(T,T)\dg\right)\\
&=&f^2l^2\left(h^g(\dg)-l^{-2}\nabla^g_{\dg}\left(\frac{\ddg}{f^2}-2\frac{f'}{f^3}\dg\right)-\frac{f'}{f}\theta-\frac12l^{-2}g(T,T)\dg\right)\\
&=&f^2\left((h^g(\dg))^\sharp-\nabla^g_{\dg}\left(\frac{\ddg}{f^2}-2\frac{f'}{f^3}\dg\right)-\frac{f'}{f}T-\frac12\frac{g(\ddg,\ddg)}{f^4}\dg\right).
\eea
Developing this expression using \eqref{q3} and the expressions for $f'$ and $f''$ given by \eqref{q2}, proves the desired result.
\end{proof}

\subsection{Tensorial invariants of a conformal embedding}
In the sequel,
$N$, of dimension $n$, will be an embedded submanifold of a conformal
manifold $(M,c)$, of dimension $m>n$, and we will consider $N$ as a
conformal submanifold of $M$ endowed with the induced conformal
structure still denoted by $c$. The cases when $n$ or $m$ are
less than $3$ will be specified explicitly. 

We denote by $\nu$ the normal bundle
to $N$, therefore $$TM|_N=TN\opr\nu.$$
We will denote by greek letters $\zeta,\xi,$ etc. normal vectors to $N$,
and by capital roman letters $X,Y,$ etc. tangent vectors to $N$. For a
vector $A\in T_XM$, $x\in N$, the components of $A$ in $T_xN$ and in
$\nu_x$ will be denoted by $A^N$, resp. $A^\perp$.

For a given Weyl structure $\nb$ on $M$,  the tangential component of $\nb_XY$, for
$X,Y\in TN$ defines a Weyl structure on $(N,c)$, called the {\em
  induced} Weyl structure and denoted by $\nb^N$:
$$\nb^N_XY:=(\nb_XY)^N.$$ 
The {\em normal} component of $\nb_XY$ (for $X,Y\in TN$) is the {\em fundamental form}
of $N\subset M$ associated to $\nb$, denoted by $B^\nb$:
$$B^\nb:\ot^2TN\ra\nu,\quad B^\nb(X,Y):=(\nb_XY)^\perp.$$
For another Weyl structure $\nb'=\nb+\tilde\theta$, for $\theta$ a
1-form, we easily get
\be\label{B}B^{\nb'}-B^\nb=-\theta^\perp\cdot c,\ee
therefore we have the following well-known fact:
\begin{prop}\label{Binv}The trace-free part $B_0\in S^2_0(T^*N)\ot\nu$
  of the fundamental form of a submanifold of a conformal manifold $M$
  depends only on the conformal structure $c$ and the embedding
  $N\subset M$.\end{prop}
A point of $N$ where $B_0$ vanishes is called {\em umbilic},
and if $B_0\equiv 0$, then $N$ is a {\em totally umbilical} submanifold. 
For example, every curve in a conformal manifold is trivially totally
umbilical.
In general, (\ref{B}) immediately implies
\begin{prop}A conformal submanifold $N\subset M$ is totally umbilical
  iff there exists a Weyl structure $\nb$ on $M$, defined along $N$, such
  that $N$ is $\nb$-totally geodesic.\end{prop}
The trace-part of the fundamental form is the {\em mean curvature}
$$H^\nb\in\nu^*\simeq
L^{-2}\ot\nu,\
H^\nb(\xi):=-\frac{1}{n}\tr\left((\nb_\mathbf{\!\!.\;}\xi)^N\right),$$  
such that we have the following decomposition of $B^\nb$:
\be\label{fform}
B^\nb=H^\nb\cdot c+B_0.\ee
By \eqref{B}, the mean curvature transforms according to the rule
\beq\label{mc}H^{\nb'}-H^\nb=-\theta^\perp.
\eeq
A Weyl connection, resp. metric on $(M,c)$ for which the mean curvature of
$N$ vanishes is called {\em adapted} to the embedding $N\subset
M$. This property makes sense pointwise: recall
that, in general, a conformal connection on a conformal manifold
$(M,c)$ is a section in an affine bundle (the associated vector bundle
is $\Lb^1M\ot\co(M)$). A Weyl structure is a section in an affine
subbundle, called $\weyl(M)$ (the associated vector subbundle is the
one of the $1$--forms with values in $\co(M)$ of type $\tilde\te$,
thus it is isomorphic with $\Lb^1M$). 
\begin{prop}\label{adpt} Let $N\subset (M,c)$ be a submanifold in a conformal
  manifold. Then the {\em adapting map} 
$$\phi^M_N:\weyl(N)\ra \weyl(M)|_N,$$
that extends each (pointwise) Weyl structure on $N$ at $x$ to an adapted
Weyl structure on $M$ at $x$ is a right inverse to the {\em
  restriction map}
$$\mathrm{rest}^M_N:\weyl(M)_N\ra \weyl(N).$$
Therefore, the image of the adapting map, $\weyl(M)^N$, is
isomorphic to $\weyl(N)$ and the quotient
$\weyl(M)|_N/\weyl^N$ is canonically isomorphic to the normal bundle
$\nu$.\end{prop}  
The proof is a direct consequence of \eqref{B}. For a result on
adapted ambient metrics, a local viepoint is necessary:
\begin{prop}\label{transcx} Let $N\subset (M,c)$ be an embedded
  submanifold in a conformal manifold of dimension $m\ge 2$, and let
  $H$ be a section in the dual normal bundle $\nu^*$. For each
  metric $g^N$ on
  $N$, there exists a metric $g^M$ on $M$, inducing $g^N$, such that
  its Levi-Civita connection has $H$ as the mean curvature of $N$. Any
  other such metric ${g^M}'$ satisfies
\be\label{gM'}{g^M}'=e^{2f}g^M,\mbox{ such that } f|_N\equiv 1 \mbox{
  and } df|_N\equiv 0.\ee\end{prop}
\begin{proof} I
It is easy to choose $g^M$ such that it restricts to $N$ as $g^N$, and
denote by $\te$ the difference between $H$ and the mean curvature of
$N$ with respect to $g^M$. We need to show that there exists a
function $f:M\ra\R$ such that\bi
\item $N=\{x\in M\ |\ f(x)=0\}$
\item $df|_N=\te$.\ei
On a chart
  domain $U\simeq \R^m$, where $U\cap N\simeq \R^n\subset\R^m$ as the
  zero set of the coordinates $y_1,...,y_{m-n}$: then a section
  $\te$ of $\nu^*|_{U\cap N}$ is given by
$$\te(x)=\sum_{i=1}^{m-n}a_i(x)dy_i,\ \forall x\in\R^n\simeq U\cap N$$
so the function $f^U:U\ra\R$ defined by 
$$f^U(x,y):=\sum_{i=1}^{m-n}a_i(x)y_i$$ 
satisfies $f^U|_{U\cap N}\equiv 0$ and
$df^U|_{U\cap N}=\te$.

Glueing such local functions using a covering of $N$ by contractible
charts and an underlying partition of unity yields a function $f:M\ra
\R$ such that
$f|_N\equiv 0$ and $df|_N=\te$, as required, thus the metric
$g^M:=e^{-2f}g$ restricts to $g^N$ on $N$ and its Levi-Civita
connection $\nb^M=\nb^g+\widetilde{df}$ induces the required mean curvature.   
\end{proof}
 In the case of a curve in $(M,c)$, Propositions \ref{adpt} and
 \ref{transcx} show that 
\begin{cor}\label{adapt} Every embedded curve $C$ in a conformal-Riemannian manifold
  $(M,c)$ is $\nb$--geodesic for some suitable ({\em adapted}) Weyl structure on
  $M$. The restriction to $TM|_C$ of this Weyl structure is determined
  by a connection $\nb^C$ on $C$ (induced by $\nb$). If $\nb^C$ is
  defined by a parametrization of $C$, $\nb$ can be chosen to be the
  Levi-Civita connection of a metric on $M$.\end{cor}
In other words, every parametrized curve is a geodesic for a suitable metric.
\begin{rem} In the pseudo-Riemannian conformal setting, this corollary
  still holds for embedded curves that are nowhere
  light-like.\end{rem}
We remark the following, concerning the
  connection induced on the normal bundle of $N$:
\begin{prop}\label{cxnorm} Let $\nb$ be a Weyl structure on $M$ and
  $\nb^N$ the induced Weyl structure on the submanifold $N$. Denote by
  $\nb^\perp$ the connection induced by $\nb$ on the normal bundle
  $\nu $  of $N$ in $M$:
\be\label{cxn}\nb^\perp_X\xi:=(\nb_X\xi)^\perp,\ \forall X\in TN.\ee
Then $\nb^\perp$ depends only on $\nb^N$. Moreover, the connection
$\nb^{\perp,0}$ induced on the weightless normal bundle
  $\nu^0:=\nu\otimes L^{-1}$ depends only on the embedding of $N$ in
  the conformal manifold $(M,c)$. In
  particular, its curvature 
\be\label{kap} \kappa\in\Lb^2\ot\End^{skew}(\nu)\ee
is a tensor that depends on the conformal embedding alone.\end{prop}
\begin{proof} Let $\nb'=\nb+\tilde\te$, with $\te$ a $1$--form on $M$.
  Then ${\nb'}^\perp=\nb^\perp+\te\ot\idt$, because the endomorphism of
  $TN^\perp$ induced by $\te\we X$ is trivial, for all $X\in TN$. The
  connection ${\nb'}^{\perp,0}$ induced on $L^{-1}\ot\nu$ coincides
  thus with $\nb^{\perp,0}$. 
\end{proof}

We have seen before that the normal component of the ambient
Schouten-Weyl tensor on an embedded curve (which is always totally
umbilical), for an adapted Weyl structure, is an invariant that
vanishes if and only if the curve is a conformal (resp. M\"obius)
geodesic: this invariant is the normal component of the conformal
acceleration, see Lemma \ref{accel}. We generalize this fact and define the
{\em mixed Schouten-Weyl tensor of an embedding} as follows:
   \begin{prop}\label{mixt}
Let $N\subset M$ such that $m=\dim M\ge 3$ or $m=2$ and $M$ is endowed
with a fixed M\"obius structure $\M$. The {\em mixed Schouten-Weyl tensor} 
\be\label{mix}\mx:TN\ra\nu^*,\ \mx_X(\xi):=h^M(X,\xi)-\left(\nb_XH^\nb\right)(\xi)+
\frac1{n-1}(\delta^\nb B_0)(X)(\xi)\ee
is independent of the Weyl structure $\nb$ on $M$ used to compute the
Schouten-Weyl tensor $h^M:=h^\nb$, the mean curvature $H^\nb$ and the
codifferential  
$$(\delta^\nb B_0)(X):=\tr^c_{TN}\left((\nb_\cdot
  B_0)(X,\cdot)\right),$$ 
Here $\nb H^\nb$ and $\nb B_0$ are the covariant derivatives induced
by $\nb$ on $\nu$, resp. on $S^2_0N\ot\nu$. The convention for $n=1$
is to omit the term in $\delta^\nb B_0$ ($B_0$ is trivial anyway) in \eqref{mix}.
\end{prop}
\begin{proof} 
Proposition \ref{cxnorm} implies that, for two Weyl structures $\nb$,
$\nb'=\nb+\tilde\te$ that extend to $M$
the same Weyl structure $\nb^N$ on $N$ (thus $\te$ vanishes on
$TN$), the connection used to compute the covariant derivative of the mean curvature does not change
(but the mean curvature changes according to \eqref{mc}). Moreover,
the codifferential of $B_0$ (which is a conformal invariant itself)
does not change at all. 

By taking the difference of the tensors $\mx$ computed using $\nb$,
resp. $\nb'$ in \eqref{mix}, the terms in the codifferential cancel,
and from the terms in $H$, only a term in $\nb(H^{\nb'}-H^\nb)$
remains, and it is equal to $-\nb\te$, according to \eqref{mc}.

On the other hand, using \eqref{scht} for the vectors $X\in TN$, $\xi\in\nu$,
we get 
$$(h^{\nb'}-h^\nb)(X,\xi)=-\nb_X\te(\xi),$$
 and this cancels with the term in $\nb
H$ as seen above.

Therefore, in computing $\mx$ via \eqref{mix}, we can choose Weyl
structures adapted to the embedding. The expression then simplifies
as the terms in $H$ disappear.

Let $\nb$ and $\nb'$ be adapted Weyl structures on $M$ that
  satisfy \eqref{e1}, such that $\te|_N$ is purely tangential. Using
  \eqref{scht} and \eqref{mc}, we compute
\be\begin{split}
(\delta^{\nb'} B_0)(X)-(\delta^\nb
B_0)(X)&=\tr^c_{TN}\left((\tilde\te_\cdot
  B_0)(X,\cdot)\right)-\tr^c_{TN}B_0(\tilde\te_\cdot X,\cdot)-\tr^c_{TN}B_0(X,\tilde\te_\cdot \cdot) \\
&=(n-1) B_0(X,\te)=(n-1) B^\nb(X,\te)\\
&=(n-1)(\nb_X\te)|_{\nu^*},\label{delt}
\end{split}\ee
because the first term in the right hand side of the first line is equal to $B_0(X,\te)$, the second vanishes
and the third is equal to $(n-2)B_0(X,\te)$. Note that the result is
indeed a section of $\nu\ot L^{-2}\simeq \nu^*$.

The claim follows now from \eqref{delt} and \eqref{scht} (for $n=1$,
the equation \eqref{delt} is not needed, the claim follows directly
from the reduction to adapted connections and \eqref{scht}).
\end{proof}
Of course, in a conformally flat space (locally isomorphic to the
M\"obius sphere), a totally umbilical submanifold of dimension at least
2 is a piece of a sphere, so the embedding is conformally equivalent
to $\R^n\subset \R^m$. This means that the mixed Schouten-Weyl tensor
automatically vanishes for totally umbilical submanifolds in
a conformally flat space. The case where the ambient space is curved
is different, as the following example shows:
\begin{example}
Cf. \cite{cp}, a {\em conformal product} is a conformal manifold $(M,c)$
such that a Weyl structure $\nb$ exists, that preserves  an orthogonal
pair of foliations, thus locally $M\simeq M_1^{m_1}\x M_2^{m_2}$. A simple example
of a conformal product is given by the conformal class of a Riemannian
product (with $\nb$ equal to the Levi-Civita connection of the product
metric). However, if the conformal product is {\em non-closed} (the
Weyl structure $\nb$ does not (even locally) preserve any metric,
i.e. its Faraday form $F$ is non-vanishing, then from Lemma 6.1
\cite{cp} we know that the Schouten-Weyl tensor of $\nb$, computed for
$X_1,X_2$ each tangent to one of the foliations, has the following
expression:  
\be\label{61}h^\nb(X_1,X_2)=\frac12\left(\frac{m_1-m_2}{m_1+m_2-2}-
  1\right)F(X_1,X_2).\ee  
Therefore, the leaves of each foliation are totally umbilical
(because they are totally geodesic w.r.t. $\nb$), but the mixed
Schouten-Weyl tensor is non-zero if $m_2>1$.\end{example}
Here we see that, for $m_2=1$, the mixed Schouten-Weyl tensor of the
first foliation vanishes. This is a general fact:
\begin{prop}\label{mu0} Let $(M,c)$ be a conformal manifold of dimension at least 3. Then the mixed Schouten-Weyl tensor of any hypersurface $N\subset
  M$ is identically zero.\end{prop}
\begin{proof} This is a consequence of the Gauss and Codazzi
  equations: using the notations from Proposition \ref{mixt}, we have
 \bea c(R^{\nb^M}_{X,Y}Z,\xi)&=&c(\nb_X B^{\nb^M}(Y,Z)- \nb_Y
 B^{\nb^M}(X,Z),\xi),\\
c(R^{\nb^M}_{X,Y}\xi,\zeta)&=&c(R^{\nb^\perp}_{X,Y}\xi,\zeta)+
c(B^{\nb^M}(X)(\xi),B^{\nb^M}(Y)(\zeta))- 
c(B^{\nb^M}(Y)(\xi),B^{\nb^M}(X)(\zeta)).\eea
Supposing that $\nb^M$ is adapted and metric (so we have no Faraday form), (so $B^{\nb^M}=B_0$) and taking the
trace on $TN$ in the first equation in $Y$ and $Z$, we obtain
$$\ric^{\nb^M}_0(X,\xi)=-\delta^{\nb^M}B_0(X)(\xi),$$
because
$\ric^{\nb^M}(X,\xi)=\tr_{TM}R^{\nb^M}_{\cdot,X}\xi=\tr_{TN}R^{\nb^M}_{\cdot,X}\xi$,
if $TN^\perp$ is generated by $\xi$. 
For the Schouten-Weyl tensor of $\nb^M$, this implies
$$h^{\nb^M}(X,\xi)=-\frac1{m-2}\delta^{\nb^M}B_0(X)(\xi).$$
Using $n=m-1$ (as $N$ is a hypersurface), then \eqref{mix} implies
$\mu\equiv 0$.\end{proof}
\begin{rem} Note that for
 $n=1$ and $m=2$, the computations above do not apply; in fact, for a
 curve $N$ in a M\"obius surface $(M,c,\M)$, the mixed Schouten-Weyl
 tensor is simply the normal part of the conformal (or rather
 M\"obius) acceleration (Lemma \ref{accel}), and thus vanishes if and
 only if the curve is a M\"obius geodesic. \end{rem}

Another invariant derived from the Schouten-Weyl tensor is the {\em
  relative Schouten-Weyl tensor of an embedding}:
\begin{prop}\label{rel}
Let $(M^m,c)$ be a conformal manifold of dimension $m\ge 3$ or a
M\"obius surface. Let $N^n\subset M^m$ such that $n=\dim N\ge 3$, or $N$ is a M\"obius
surface with the same underlying conformal
structure as the one induced by $M$, or a Laplace curve, and let $\nb$ be any Weyl
structure on $M$. We denote as before by $h^M$ and $H^\nb\in\nu^*$ the
Schouten-Weyl tensor of $\nb$ on $M$ and the mean curvature of $N\subset M$
with respect to $\nb$. Let $h^N$ be the Schouten-Weyl tensor of $N$
corresponding to the Weyl structure on $N$ induced by $\nb$ (see
Remarks \ref{schtm} and \ref{schtl} for the low-dimensional cases). Then the
{\em relative Schouten-Weyl tensor} $\rs\in S^2N$, defined by 
\be\label{rels}\rs(X,Y):=(h^M-h^N)(X,Y)+\frac{1}{2}c(H^\nb,H^\nb)c(X,Y)
+H^\nb(B_0(X,Y)),\ \forall X,Y\in TN,\ee
is independent of the Weyl structure $\nb$.  
\end{prop}
\begin{proof}
As in the proof of Proposition \ref{mixt}, we treat first the case of
two Weyl structures $\nb$, $\nb'=\nb+\tilde\te$ on $M$ that restrict
to the same connection $\nb^N$ on $N$, i.e., $\te|_{TN}\equiv 0$.

In the expressions \eqref{rels} corresponding to $\nb$, resp. $\nb'$,
the terms in $h^N$ are the same, and the difference of the terms in $h^M$
is, cf. \eqref{scht}:
\be\label{relh}h^{\nb'}-h^\nb=-\nb\te+\te\ot\te-\frac12
c(\te,\te)c,\ee
where the middle term of the right hand side, restricted to $TN\ot TN$, vanishes. The term in
$\nb\te$ is, in fact, an expression in the fundamental form alone:
$$-\nb_X\te(Y)=\te(\nb_XY)=c(\te,H^\nb)c(X,Y)+\te(B_0(X,Y)),\ \forall
X,Y\in TN,$$
because $\te(X)=\te(Y)=0$. We conclude
$$(h^{\nb'}-h^\nb)|_{TN\ot TN}=c(\te,H^\nb)c+\te(B_0)-\frac12
c(\te,\te)c.$$
It is clear from \eqref{mc} that the right hand side of the equation
above cancels the difference
$$\left(\frac12 c(H^{\nb'},H^{\nb'})c + H^{\nb'}(B_0)\right)-
\left(\frac12 c(H^{\nb},H^{\nb})c + H^{\nb}(B_0)\right).$$
This means that, in the expression \eqref{rels}, we can restrict to
adapted Weyl structures.

Let $\nb$ and $\nb'=\nb+\tilde\te$ be adapted Weyl structures. Then
the expressions \eqref{rels} for $\nb$, resp. $\nb'$, simplify and
their difference is given by
$$(h^{{\nb'}^M}-h^{\nb^M})|_{TN\ot TN}-(h^{{\nb'}^N}-h^{\nb^N}).$$
Applying \eqref{scht} for both the pairs $({\nb'}^M,\nb^M)$ and (the
induced connections) $({\nb'}^N,\nb^N)$, we obtain the same result,
because $\te$ is tangential to $N$.

This shows the independence of $\rs$ of the Weyl structure $\nb$, as claimed.
\end{proof}
Like $\mx$, $\rs$ vanishes on a totally umbilical submanifold of a
conformally flat space. However, in the curved setting, the vanishing
of $B_0$, and even of $\mx$, does not imply the vanishing of $\rs$:
\begin{example} Let $M$ be the Riemannian product $M_1\times M_2$, and
  let $N:=M_1\times \{x\}$, where $x\in M_2$ (we assume $\dim M_i>2$,
  $i=1,2$). $N$ is totally geodesic in $M$, for
  the product metric (which is therefore adapted to the embedding),
  hence $B_0=0$, and, also for the product
  metric, $\mx$ vanishes. On the other hand, even if the Ricci tensor
  of $N$ is the restriction to $N$ of the Ricci tensor of $M$ (again
  for the product metric), this property does not hold in general
  for the Schouten-Weyl tensors, because the normalizations depend on
  the scalar curvatures of $M$, resp. $N$, and on the dimensions $m$
  and $n$. For example, if $M_1$ is an
  Euclidean space and $M_2$ is a
  round sphere, $h^{M_1\times M_2}|_N\ne h^{M_1}=0$.\end{example} 

One can ask whether the Gauss and Codazzi type equations imply any
further relation between the invariants $B_0$, $\nb^\nu$, $\mu$ and
$\rho$ of a conformal embedding $N\subset M$ (besides the vanishing of
$\mu $ for hypersurfaces, cf. Proposition \ref{mu0}). The following Theorem
proves the contrary: on any conformal manifold (or M\"obius surface,
or Laplace curve) one can prescribe all the invariant tensors defined
in this section: $B_0$, $\kappa$, $\mu$\footnote{for an embedding of
  $N$ as a hypersurface, $\mu$ has to be zero, cf. Proposition \ref{mu0}} and $\rho$ and realize them by
a suitable conformal embedding:
\begin{thrm}\label{existt} Let $(N,g)$ be a Riemannian manifold and
  let $\nu$ be a vector bundle over $N$ endowed with a metric $g^\nu$ and a
  connection $\nb^\nu$. If $\dim N=1$ or $2$, let $N$ have a fixed
  Laplace, resp. M\"obius
  structure, given by the Schouten-Weyl tensor $h$ (associated to the
  Levi-Civita connection $\nb$ of $g$). 

Let $B_0\in\ci(S^2_0N\ot \nu)$
  and $\rho\in\ci(S^2N)$ be given
  tensors; if $\rank\,\nu>1$ or if $\dim N=1$ and
  $\rank\,\nu=1$, let
  $\mu\in\ci(T^*N\ot\nu^*)$  be a given arbitrary section, otherwise
  let $\mu$ be the zero section of this bundle. 

Then there exists a metric $\tl g$ on the total space $M$ of
  $\nu$, and, for $\dim N=\rank\,\nu=1$, there exists a M\"obius
  structure on $M$, such that the fundamental form of $N$ with respect
  to $\tl g$ is $B_0$, and the mixed, resp. relative Schouten-Weyl
  tensors of $N$ in $M$ are equal to $\mu$, resp. $\rho$.\end{thrm}
\begin{proof} It is enough to construct a metric $\tl g$ on a
  neighborhood of the zero section of the total
  space $M$ of the vector bundle $\nu$. We introduce the following
  tensors
$$a\in \ci(T^*M\ot \nu^*),\ b\in\ci(S^2M), \ f\in\ci(M)
$$
and we introduce the following notation: for each vector field $X$ on
$M$, we denote by $\tl X$ the horizontal (for the connection
$\nb^\nu$) lift on $M$ of $X$. On the other hand, for each local section
$\zeta$ in $\nu$, we denote by $\tl \zeta$ the vertical (local) vector
field on $M$ defined by $\zeta$.

We have then, for each point $\xi\in M$, the following relations:
\be\label{brack}[\tl X,\tl Y]_\xi=\wtl{[X,Y]}_\xi-\wtl{R^\nu_{X,Y}\xi},\ [\tl
X,\tl\zeta]_\xi=\wtl{\nb^\nu_X\zeta},\ [\tl\zeta,\tl\eta]_\xi=0,\ee
for each vector fields $X,Y$ on $M$ and sections $\zeta,\eta$ in $\nu$. 

We define the following symmetric bilinear form $\tl g$ on $M$: in a
point $\xi\in M$ we define 
\beay\label{gtl}\tl g(\tl X,\tl Y)_\xi&:=&g(X,Y)
-2g^\nu(B_0(X,Y),\xi)+b(X,Y)\|\xi\|^2,\nonumber\\
\tl g(\tl X,\tl \zeta)_\xi&:=& \|\xi\|^2a(X,\zeta),\\
\tl g(\tl \zeta,\tl \eta)_\xi&:=& g^\nu(\zeta,\eta) (1+f\|\xi\|^2)
.
\nonumber\eeay
Here, $\|\xi\|^2:=g^\nu(\xi,\xi)$.
It is clear that $\tl g_\xi$ is positive definite for $\xi$ in a
neighborhood of the zero section of $\nu$.

We can compute the curvature of the Levi-Civita connection $\tl\nb$ of $\tl g$
on the points of $N\subset M$. For this, we first compute the
covariant derivatives of vector fields of type $\tl X,\tl\zeta$ (as
defined above) up to order 1 in the vertical (i.e., in the direction
of the fibers of $\nu$) directions:
\beay\label{a1}\tl\nb_{\tl X}\tl Y&=&\wtl{\nb_XY}+B_0(X,Y)-b(X,Y)\xi
-\frac12 R^\nu_{X,Y}\xi \nonumber\\
&&
-\nb_XB_0(Y,\cdot)(\xi)-\nb_YB_0(X,\cdot)(\xi)+\nb_\cdot B_0(X,Y)(\xi)
+\mathcal{O}(\|\xi\|^2),\nonumber \\
\tl\nb_{\tl\zeta}\tl X&=&-B_0(X,\cdot)(\zeta)+b(X,\cdot)g^\nu(\zeta,\xi)
+\frac12 h(R^\nu_{X\cdot}\xi,\zeta)\nonumber\\
&&+a(X,\cdot)g^\nu(\zeta,\xi)-
a(X,\zeta)\xi+\mathcal{O}(\|\xi\|^2),\\
\tl\nb_{\tl X}\tl\zeta&=&-B_0(X,\cdot)(\zeta)+b(X,\cdot)g^\nu(\zeta,\xi)
+\frac12 g^\nu(R^\nu_{X\cdot}\xi,\zeta)\nonumber\\
&&+a(X,\cdot)g^\nu(\zeta,\xi)-
a(X,\zeta)\xi+\nb^\nu_X\zeta+\mathcal{O}(\|\xi\|^2),\nonumber \\
\tl\nb_{\tl\zeta}\tl \eta&=& 
 (g^\nu(\xi,\zeta)\eta+g^\nu(\xi,\eta)\zeta-g^\nu(\zeta,\eta)\xi) f
 +
a(\cdot,\eta)g^\nu(\xi,\zeta)+a(\cdot,\zeta)g^\nu(\xi,\eta)
+\mathcal{O}(\|\xi\|^2).\nonumber \eeay
We compute then the curvature of $\tl\nb$ on $N$ and the Ricci tensor
and the scalar curvature is then, for $x\in N\subset M$:
\beay\label{rictl}
\wtl{\ric}(\tl X)_x&=&\wtl{\ric(X)}-(m-n)b(X,\cdot)-\delta^\nb
B_0(X)-(m-n-1)a(X,\cdot)\nonumber \\
\wtl{\ric}(\tl\zeta)_x&=&-\delta^\nb B_0(\cdot)(\zeta)-
(m-n-1)a(\cdot,\zeta)-\tr (b)\zeta -2(m-n-1)f\zeta.
\\
\wtl\scal&=& \scal -2(m-n)\tr (b) -2(m-n-1)(m-n)f
.\nonumber\eeay
If $m-n-1>0$, the map associating the mixed Schouten-Weyl tensor $\mu$
of the embedding $N\subset M$ to a tensor $a$ is an invertible affine map.

If $m\ge 2$, the map associating the trace of the relative
Schouten-Weyl tensor $\rho$ to $f$ is also an invertible affine map.

Once $a,f$ and the trace part of $b$ are fixed, if $m\ge 3$, the map
associating the trace-free relative Schouten-Weyl tensor $\rho_0$ to
$b_0$ (the trace-free part of $b$) is also an invertible affine map. 

In conclusion, for any given tensors $\mu$ and $\rho$ as above, by
choosing $a,b$ and $f$, the mixed and relative Schouten-Weyl tensors of
$N\subset (M,[\tl g])$ are precisely $\mu$, resp. $\rho$.

The same conclusion holds in the special case $m=2$ (when $M$ is a
M\"obius surface) and $n=1$ ($N$ is a Laplace curve), but in this case
$\tl g$ can be arbitrarily chosen; the mixed Schouten-Weyl tensor $\mu$ and the relative Schouten-Weyl
tensor $\rho$ determine (and are determined by) a choice of a M\"obius
structure on $M$, or, equivalently, by a choice of a trace-free symmetric
Schouten-Weyl tensor $\tl h_0$. Therefore, any given tensors
$\mu$ and $\rho$ can be realized as the corresponding invariants of an
embedding of the Laplace curve $N$ in some M\"obius surface $M$.
\end{proof}


\subsection{Induced M\"obius and Laplace structures on submanifolds}

\subsubsection{Induced M\"obius structure}
Let $N^n\subset M^m$ be a submanifold in the conformal manifold $(M,c)$ with 
$m>n\ge 2$. We want to construct an {\em induced} M\"obius structure
on $N$ (see also \cite{bc}):

Take any Weyl structure $\nb$ on $(M,c)$. We try to mimic the
definition of the canonical M\"obius structure by considering the
Hessian of $\nb$ acting on vectors of $TN$ and on sections of
$L_N$. Since the usual Hessian possibly involves a covariant derivative in
normal directions, we need to modify it by suppressing this
derivative. More precisely,
we define the ``horizontal'' Hessian acting on sections of $L_N$ by
\beq\label{hsb}\overline\hs^\nb(X,Y)=\nb_X\nb_Y-\nb_{(\nb_XY)^N}=
\hs^\nb(X,Y)+\nb_{B(X,Y)},\qquad\forall X,Y\in TN.\eeq 
We define then the second order operator $P:C^\infty(L_N)\to C^\infty(T^*N\otimes
T^*N\otimes L_N)$ by
$$P(X,Y)l=\overline\hs^\nb(X,Y)l+h^\nb(X,Y)l,\qquad\forall X,Y\in TN,\ l\in
C^\infty(L_N).$$

We now study the way $P$ changes under a change of Weyl structure. 
If we replace $\nb$ by $\nb'=\nb+\tilde\theta$, the formulas for the
change of the Schouten-Weyl tensor \eqref{scht} and of the Hessian
\eqref{hess}, together with \eqref{B} yield 
\bea P'(X,Y)l-P(X,Y)l&=&c(X,Y)(\nb_\theta l+\frac12c(\theta,\theta)l)+\nb'_{B'(X,Y)}l-\nb_{B(X,Y)}l\\
&=&c(X,Y)(\nb_{\theta^N}
l+\frac12c(\theta,\theta)l
-c(\theta,\theta^\perp)l)+\theta(B(X,Y))l.
\eea

Let $P_0$ denote the trace-free part of $P$: $P_0(X,Y)l=P(X,Y)l-\frac1nc(X,Y)\tr_{TN}(Pl)$. The previous relation yields
\beq\label{a1}P'_0(X,Y)l-P_0(X,Y)l=\theta^\perp(B_0(X,Y))l.\eeq
We conclude:
\begin{prop}\label{mind} Let $N^n\subset (M^m,c)$ be conformal
  submanifold, with $m>n\ge 2$. 
The second order operator
$$\M^{ind}:C^\infty(L_N)\to C^\infty(T^*N\otimes T^*N\otimes L_N)$$
defined by
\be\label{emind}\M^{ind}(X,Y)l=P_0(X,Y)l+H^\nb(B_0(X,Y))l,\ee
or, equivalently, by
\beq\label{2mind}\M^{ind}_{X,Y}l:=\left(\hs^\nb_{0,N}(X,Y)l+\nb_{B_0(X,Y)}l\right)+h^\nb_{0,N}(X,
Y)+H^\nb(B_0(X,Y))l\eeq
(here the double subscript $0,N$ means the trace-free part along $TN$)
does not depend on the 
choice of the Weyl structure $\nb$, and defines a M\"obius structure
on $N$, called the {\em induced M\"obius structure} on $N$.

Moreover, if $n\ge 3$ or $n=2$ and $N$ has a M\"obius
  structure $\M^N$ compatible with $c$, we have
\be\label{mcomp} \M^N-\M^{ind}=\rs_0,\ee
where $\rs$ is the relative Schouten-Weyl tensor of the embedding
\eqref{rels}, and $\rs_0$ its trace-free part.\end{prop}
\begin{proof} The independence of $\M^{ind}$ of the Weyl structure
  $\nb$ follows directly from \eqref{a1} and \eqref{mc}.
 
We will now compute the difference between this induced M\"obius
structure and the M\"obius structure of $N$. Since the horizontal
Hessian $\overline\hs^\nb(X,Y)l$ appearing in the definition of $P$ is
just the Hessian of the induced Weyl structure on $N$ applied to $l$,
we get  
\begin{equation*}\begin{split}P(X,Y)l+&H^\nb(B(X,Y))l-(\hs^N(X,Y)l+h^N(X,Y)l)=\\&h^M(X,Y)l
-h^N(X,Y)l+H^\nb(B(X,Y))l.\end{split}\end{equation*}
The trace-free part of the left hand side equals $\M^{ind}l-\M^Nl$, whereas 
the trace-free part of the right hand side is just the trace-free part of $\rs$. This proves our claim.\end{proof}

\subsubsection{Induced Laplace structure} 
 Let $N^n\subset (M^m,c)$ be conformal
  submanifold, with $m>n\ge 1$, and if $m=2$, suppose $M$ has a
  M\"obius structure (see also \cite{bc}).

In order to define the induced Laplace structure on $N$, we take, by analogy with the definition in Proposition \ref{lapex},
the trace on $TN$ of the horizontal 
Hessian acting on sections of $L^k$ for $k=1-\frac n2$ and add $k$ times the trace over $TN$ of the 
Schouten tensor $h^\nb$ of $M$ with respect to $\nb$. The resulting
operator, say $D:C^\infty(L^k)\ra C^\infty(L^{k-2}$, defined by
$$D(l):=\tr_{TN}(\overline\hs^\nb l+k h^\nb(\cdot,\cdot) )l$$
is not independent on $\nb$. Indeed, for a new Weyl structure
$\nb'=\nb+\tilde\theta$, we get using \eqref{scht}, \eqref{hess} and \eqref{hsb}:
\bea D'(l)-D(l)&=&\tr_{TN}(\overline\hs^{\nb'} l-\overline\hs^\nb l+k h^{\nb'}(\cdot,\cdot) l-k h^\nb(\cdot,\cdot) l)\\
&=&n\nb'_{H^{\nb'}}l-n\nb_{H^\nb}l+2(k-1)\nb_{\theta^N}l+n\nb_{\theta}l\\&&+k[-\delta^N\theta+(k-2)c(\theta^N,\theta^N)+nc(\theta,\theta)]l
\\&&+k[\delta^N\theta+c(\theta^N,\theta^N)-\frac n2c(\theta,\theta)]l.
\eea
After simplifications (using \eqref{der}, \eqref{mc} and the choice $k=1-\frac n2$), we get
$$D'(l)-D(l)=knc(\theta,H^{\nb'})l+\frac{kn}2c(\theta^\perp,\theta^\perp)l.$$
On the other hand, by \eqref{mc} the expression $c(H^\nb,H^\nb)$ transforms according to the rule
$$c(H^{\nb'},H^{\nb'})-c(H^{\nb},H^{\nb})=-2c(\theta^\perp,H^{\nb'})-c(\theta^\perp,\theta^\perp).$$
 We have proved the following:

%

\begin{prop}\label{ind}  Let $N^n\subset (M^m,c)$ be conformal
  submanifold, with $m>n\ge 1$, and if $m=2$, suppose $M$ has a
  M\"obius structure. The following operator is well-defined, and is
  independent of the Weyl structure $\nb$:
$$\Lp^{ind}:\ci(L_N^{1-n/2}\ra\ci(L_N^{-1-n/2}),$$
\be\label{lind}\Lp^{ind}l:=\tr _{TN}\left(\overline\hs^\nb l\right)+ 
k \tr _{TN}h^\nb l + \frac{kn}{2}c(H^\nb,H^\nb)l.\ee 
It is called the
{\em induced Laplace structure} on $N$.
\end{prop}

There is a geometric interpretation of the induced M\"obius operator: Suppose 
$l$ is a section in $L$ over the submanifold $M$, that does not vanish. 
The square of $l$ is the a metric $g^N$ on $N$ in the conformal class $c$. 
It is actually a metric on $M$ as well, but it is only defined {\em along N}.
We can, however, extend it to a metric $g^M$ on (a neighborhood of $N$ in) $M$, 
and we call such an extension {\em minimal} if $(N,g^N)$ is a minimal
submanifold of $(M,g^M)$. We have then a corresponding {\em minimal} extension 
$l^{min}$ of $l$ as a section of $L$ over $M$ and it holds:

\begin{prop}\label{Mob-geom} For a minimal extension of a non-vanishing 
section $l$ of $L$ over $N$ to a section $l^{min}$ of $L$ over $M$, we have
$$\M^{ind} l=\M^M l^{min}.$$
\end{prop}

\section{Conformal geodesic submanifolds} In Riemannian geometry, an
embedding $N\subset (M,g)$ is 
{\em totally geodesic} if and only if one of the following equivalent
properties hold:
\bi\item The fundamental form $B$ of $N\subset M$ vanishes
\item For every $M$-geodesic $\gamma:I\ra M$ ($I\subset \R$ contains a neighborhood
  of $0$) that is defined by initial conditions
\be\label{init}\gamma(0)=x,\ \dg(0)=X,\ee
with $x\in N$ and $X\in T_xN$, there exists $\e>0$ such that
$\gamma(-\e,\e)$ is contained in $N$
\item All $N$-geodesics are also $M$-geodesics.\ei
In conformal geometry, these three properties give rise to different
notions: The first is the one of a {\em totally umbilical}
submanifold, that we repeat here:
\begin{defi} A submanifold $N^n\subset (M^m,c)$ in a conformal
  manifold is {\em totally umbilical} if and only if the trace-free
  fundamental form $B_0$ of the embedding vanishes identically.
\end{defi}
The second notion gives rise to the notion of a {\em weakly geodesic
  submanifold}:
\begin{defi} A submanifold $N^n\subset (M^m,c)$ in a conformal
  manifold with $\dim M\ge 3$ or a M\"obius surface is {\em weakly
    geodesic} if an only if, for every initial conditions $x\in N$,
  $X\in T_xN$ and $D\in \weyl_x^N$ for the conformal geodesic equation
  \eqref{cg} in $M$:
\be\label{cginit} \gamma:I\ra M,\ \gamma(0)=x,\ \dg(0)=X,\ \nb_x=D,\ee
where $I$ is an interval that contains a neighborhood of $0$, there
exists $\e>0$ such that
$\gamma(-\e,\e)$ is contained in $N$.\end{defi}
Finally, the strongest notion of a conformal geodesic submanifold is
\begin{defi} Let $N^n\subset (M^m,c)$ be a submanifold of dimension at least
  3, or $\dim N=2$ and $N$ has a compatible M\"obius structure, or
  $\dim N=1$ and $N$ has a compatible Laplace structure. Suppose $\dim
  M\ge 3$ or $\dim M=2$ and $M$ has a M\"obius surface compatible with
  $c$. $N$ is then {\em strongly
    geodesic} if an only if every $N$--conformal (or M\"obius, or
  Laplace) geodesic is also an $M$--conformal or M\"obius
  geodesic.\end{defi}
\begin{thrm}\label{implic} For a submanifold $N\subset M$, where $(M,c)$ is a
  conformal manifold of dimension $m\ge 3$ or a M\"obius surface, and
  $N$ has, if $\dim N=1$ or $2$, a Laplace, resp. M\"obius structure compatible with $c$,  the following implications hold:
\be\label{impl}\mbox{strongly geodesic}\Rightarrow \mbox{weakly
  geodesic}\Rightarrow \mbox{totally umbilical},\ee
and the converse implications hold if $(M,c)$ is conformally flat (for
$m\ge 3$), resp. $(M,c,\M)$ is M\"obius flat (for $m=2$). 
Moreover, these three properties of the embedding $n\subset M$ are
equivalent to the vanishing of the following invariant tensors:
\bi\item $N\subset M$ is totally umbilical $\iff B_0\equiv 0$
\item $N\subset M$ is weakly geodesic $\iff B_0\equiv 0$ and
  $\mu\equiv 0$
\item $N\subset M$ is strongly geodesic $\iff B_0\equiv 0$,
  $\mu\equiv 0$ and $\rho\equiv 0$.\ei \end{thrm}
\begin{proof} Let $N\subset M$ be weakly geodesic. Fix $x\in N$
  and fix  an adapted Weyl structure $\nb$ at $x$. Consider now all conformal
  geodesics that satisfy the initial conditions \eqref{cginit} for all
  $X\in T_xN$. These curves all lie in $N$, thus in particular we have 
$$\nb_XX=0, \ \forall X\in T_xN.$$
This means, in particular, as $x$ was chosen arbitrarily, that $N$ is
totally umbilical. Moreover, for each conformal geodesic satisfying
the initial conditions \eqref{cginit}, $h^\nb(X,\xi)=0$ for some local
extension of the Weyl structure $\nb$ and for any normal vector
$\xi\in\nu_x$. In particular $\mu_x=0$, which proves the claim (2).

Suppose now that $N$ is strongly geodesic. In particular, all the
conformal geodesics $\gamma^X$ satisfying the initial conditions \eqref{cginit},
for any $x\in N$ fixed and a fixed connection $D$ on $N$, are also conformal
geodesics in $M$. However, the initial conditions satisfied by
$\gamma^X$ as a conformal geodesic on $M$ are
\be\label{cgM}\gamma^X(0)=x,\ \dg^X(0)=X\mbox{ and
}\nb^{\gamma^X}_x=\nb^X\in\weyl^M(x),\ee
where we don't know yet what is the Weyl structure $\nb^X$, except
that it is adapted to the parametrized curve $\gamma$ in $M$, thus it
has to be an extension of the connection $D$ on $N$, with respect to
which $\gamma$ is geodesic in $N$. We will
simply relate it to the adapted extension $\nb$ of $D$ as
$$\nb^X=\nb+\tilde\te^X,\mbox{ where }\te^X\in \nu^*.$$
The mean curvature of $\nb^X$ at $x$ is then $-\te^X$ \eqref{mc}. On
the other hand, $(\nb^X_XX)_x=0$, because $\gamma^X$ is geodesic for
$\nb^X$, therefore
\be\label{tex}-\te^Xc(X,X)+B_0(X,X)=0.\ee
Now we compute the relative Schouten-Weyl tensor $\rs$ at $x$ using the
connection $\nb^X$: 
$$\rs_x=h^{\nb^X}_x|_{T_xN\x T_xN}-h^D_x+\frac12 c(\te^X,\te^X)c -\te^X(B_0).$$
Because $\gamma^X$ is a conformal geodesic at $x$ for both $D$ and
$\nb^x$, the Schouten-Weyl tensors of $D$ and $\nb^X$ vanish in the
direction of $X$: As this holds for all vectors $X\in T_xN$, this means that $h^D_x=0$, and
\be\label{rsX}\rs(X)=\frac12 c(\te^X,\te^X)X-\te^X(B_0(X)), \ \forall
X\in T_xN,\ee
which, cf. \eqref{tex}, implies
\be\label{rhoxx}\rs(X,X)c(X,X)=-\frac12 c(B_0(X,X),B_0(X,X)),\ \forall
X\in T_xN.\ee


We compute now the mixed Schouten tensor using the connection $\nb^X$
(we can suppose $n>1$ since the $1$-dimensional case is trivial):
$$\mu(X)=h^{\nb^X}(X)|_\nu+\nb_X\te^X+\frac1{n-1}(\delta^\nb B_0)(X).$$
Of course, $h^{\nb^X}(X)|_\nu=0$, hence we obtain
\be\label{mux}\mu_x(X)=\nb_XB_0(X,X)/c(X,X)+\frac1{n-1}(\delta^\nb B_0)(X),\
\forall X\in T_xN\sm 0.\ee
Recall that the connection used to compute $\nb B_0$ and $\delta^\nb
B_0$ depends only on the restriction to $N$ of this connection
(Proposition \ref{cxnorm}). Because $\mu$ and $\delta^\nb B_0$ are
linear maps from $T_xN$ to $\nu_x^*$, the same must hold for the map
$$X\longmapsto \frac{\nb_XB_0(X,X)}{c(X,X)}, $$
i.e., there exists a tensor $a\in T^*M\ot \nu$ such that
\be\label{aa} \nb_XB_0(X,X)=a(X)c(X,X),\ \forall X\in TM.\ee
Now, the tensor $\nb B_0$ is a $3$-tensor with values in $\nu$, that
is symmetric and trace-free in the last 2 arguments. There is,
therefore, only one trace left, $\delta^\nb B_0$. The equation
\eqref{aa} shows that the symmetric part of this 3-tensor is equal to
the symmetric part of $a\ot c$, more precisely
$$\nb_X B_0(Y,Z)+\nb_YB_0(X,Z)+\nb_ZB_0(X,Y)=
a(X)c(Y,Z)+a(Y)c(X,Z)+a(Z)c(Y,X),$$
for all $X,Y,Z\in TN$. By taking the trace in $Z,Y$, we obtain
$$2\delta^\nb B_0(X)=(n+2)a(X),$$
Thus \eqref{mux} becomes
\be\label{muu}\mu=\frac{3n}{n-1}\delta^\nb B_0.\ee
We have started with one Weyl structure $D\in \weyl_x(N)$, but
the whole argument holds for any starting Weyl structure, thus the
relation \eqref{muu} must hold for $D^\eta:=D+\tl\eta$, for all
$\eta\in T^*_xN$. Equation \eqref{delt}, together with \eqref{muu}, implies then
$$3n B_0(\eta,X)=0,\ \forall X\in T_xN, \forall \eta\in T^*_xN,$$
thus $B_0\equiv 0$ and $N$ is totally umbilical. In particular,
\eqref{tex} implies that all the connections $\nb^X$ coincide with
$\nb$ at $x$ and thus, all the conformal geodesics of $N$ are
conformal geodesics for $M$ corresponding to adapted initial
conditions. In other words, $N\subset M$ is weakly geodesic.

From \eqref{muu} and \eqref{rhoxx} we also conclude that, for a strongly
geodesic submanifold $N\subset M$, $B_0$, $\mu$ and $\rho$ vanish
identically.

Conversely, if $B_0\equiv 0$ and $\mu\equiv 0$, we consider on $N$ the
following type of curves: an {\em adapted conformal geodesic} on $N$
is a smooth curve $\gamma:I\ra N$ such that, for the adapted Weyl
structure $D$ on it, $\bar h(\dg)=0$, where $\bar h:=h^\nb|_{TN\ot
  TN}$, where $\nb$ is the unique adapted extension to $M$ of the Weyl
structure $D$ (by Propsition \ref{cxnorm}, $\nb$ is also defined along
the curve $\gamma(I)$, like $D$). This third order ODE has solutions
on $N$, and because $B_0\equiv 0$ and $\mu\equiv 0$, the resulting
curves are conformal geodesics in $M$. $N$ is, thus, weakly geodesic.

If, additionally, $\rho\equiv 0$, the above construction yields
precisely the conformal geodesics of $N$, thus $N$ is strongly
geodesic in $M$, as claimed.
\end{proof}
One can ask whether the hypothesis of weak or strong geodesy can be
relaxed, for example:
\begin{defi} A submanifold $N\subset M$ for which, in each point, the conformal
geodesics of $N$ with the initial conditions \eqref{cginit} for all
$x\in N$, all $X\in T_xN$, for some $D\in\weyl_x(N)$, are
conformal geodesics in $M$, is called {\em pseudo-geodesic}.\end{defi}
It is easy to see that a pseudo-geodesic curve is automatically
strongly geodesic.

 From the proof of the above Thorem, it is
clear that a pseudo-geodesic submanifold $N$ is strongly geodesic in
$M$, provided $N$ is totally umbilical. The following example shows
that, in higher dimensions, there are non-umbilic submanifolds that
are pseudo-geodesic:
\begin{example} Let $N:=\R^2$ with the flat metric and the flat M\"obius
  structure, and consider $\nu $ the trivial vector bundle of rank $2$
  on $\R^2$. Let
  $B_0$ be defined as follows:
$$B_0(\d_x,\d_y)=V,\ B_0(\d_x,\d_x)=-B_0(\d_y,\d_y):=W,$$
for $V,W$ an orthonormal frame of $\nu$. Let $\nb^\nu$ be a unitary
connection on $\nu$ such that $V$ and $W$ are parallel. 
Define the tensor $\rho$ as follows:
$$\rho(\d_x,\d_x)=\rho(\d_y,\d_y):=-\frac12,\ \rho(\d_x,\d_y):=0.$$
Then, by
Theorem \ref{existt}, define a
metric $\tl g$ on $M\subset \R^4$ by the conditions $\mu\equiv 0$
and $\rho$ and $B_0$ as above. 

For the Levi-Civita connection $D$ on $N$ given by the flat euclidean
metric, the geodesics (and M\"obius geodesics as well)
are clearly the affine lines. Let $\nb$ be the adapted Weyl structure
on $M$ for $N\subset M$ (in fact, it is just the Levi-Civita
connection of $\tl g$, as given by Theorem \ref{existt}).

Consider now, for each $D$-parallel unitary
vector field $X^t:=\cos t \d_x+\sin t\d_y$, we consider the Weyl structure 
$$\nb^t:=\nb+\tl\te^t,\mbox{ where }\te^t:=\cos(2t) W+\sin(2t)V.$$
For future use, we compute 
\be\label{bxt}\langle X^t,X^s\rangle = \cos(t-s),\
B_0(X^t,X^s)=\te^{\frac{t+s}{2}},\ \langle\te^t,\te^s\rangle=\cos(2t-2s).\ee

For the connection $\nb^t$, we have 
$$\nb^t_{X^t} X^t=B_0(X^t,X^t)-\cos(2t) W-\sin(2t)V= 0,$$
thus $\nb^t$ is the connection adapted to the affine line generated by
$X^t$. In order to compute $h^t(X^t)$, for $h^t$ the Schouten-Weyl
tensor of $\nb^t$ on $M$, we use $\nb^t$, resp. $D$ in the defining
formulas for $\mu$ and $\rho$ and obtain
$$h^t(X^t,V)=h^t(X^t,W)=0,$$
and 
$$h^t(X^t,X^s)=\rho(X^t,X^s)-\frac12 \langle X^t,X^s\rangle+\langle
\te^t,B_0(X^t,X^s)\rangle=0.$$
 Therefore, the affine lines in the direction $X^t$ in $N$ are
 conformal geodesics in $M$, with adapted Weyl structure $\nb^t$. $N$
 is thus spanned by a large family of common conformal geodesics (for
 $N$ and $M$), however, $N$ is nowhere umbilic in $M$.
\end{example}
It is possible to construct such examples in higher dimensions as
well. On the other hand, it is clear from \eqref{rsX} that 
for $X,Y\in T_xN$, $X\perp Y$, we get
$$\te^X(B_0(X,Y))=\te^Y(B_0(X,Y)),$$
thus a 
pseudo-geodesic M\"obius surface in a 3-manifold is always totally
umbilic, thus strongly geodesic, so the example above realizes the
minimal dimensions $(m,n)$ for which a non-umbilic, pseudo-geodesic
submanifold of dimension $n$ is a $m$-dimensional conformal manifold
exists.

 It would be interesting to find for
which dimensions $(m,n)$ the pseudo-geodesic $n$-submanifolds of a
conformal $m$-manifold are automatically umbilical. 

\end{document}